\documentclass[11pt,twoside,a4paper]{amsart}
\usepackage{amssymb}
\usepackage{amsmath,amscd}
\usepackage[initials,nobysame,shortalphabetic]{amsrefs}
\usepackage{hyphenat}
\usepackage[utf8]{inputenc}
\usepackage[all,cmtip]{xy}
\usepackage{mathtools}

\usepackage{geometry}
\def\la{\langle}
\def\ra{\rangle}

\def\dbar{\bar\partial}

\def\R{{\mathbb R}}
\def\C{{\mathbb C}}

\def\Cn{\C^n}

\def\codim{{\rm codim\,}}

\def\N{{\mathbb N}}

\def\Ok{{\mathcal O}}

\def\sing{{{\rm sing}}}

\newcommand{\Com}[1]{}
\DeclareMathOperator{\Id}{Id}
\DeclareMathOperator{\supp}{supp}
\DeclareMathOperator{\ann}{ann}

\DeclareMathOperator{\depth}{depth}

\DeclareMathOperator{\Ext}{Ext}

\DeclareMathOperator{\Hom}{Hom}
\DeclareMathOperator{\im}{im}

\DeclareMathOperator{\coker}{coker}

\def\be{\begin{equation}}
\def\ee{\end{equation}}

\newtheorem{thm}{Theorem}[section]
\newtheorem{lma}[thm]{Lemma}
\newtheorem{cor}[thm]{Corollary}
\newtheorem{prop}[thm]{Proposition}

\theoremstyle{definition}

\theoremstyle{remark}

\newtheorem{preremark}[thm]{Remark}
\newtheorem{preex}[thm]{Example}

\newenvironment{remark}{\begin{preremark}}{\end{preremark}}
\newenvironment{ex}{\begin{preex}}{\end{preex}}

\numberwithin{equation}{section}

\begin{document}

\title{Explicit versions of the local duality theorem in $\C^n$}

\author{Richard L\"ark\"ang}

\date{\today}

\address{Department of Mathematics, Chalmers University of Technology and the University of Gothenburg, 412 96 Gothenburg, Sweden}

\email{larkang@chalmers.se}



\thanks{The author was supported the Swedish Research Council.}

\begin{abstract}
    We consider versions of the local duality theorem in $\C^n$.
    We show that there exist canonical pairings in these versions of the duality theorem
    which can be expressed explicitly in terms of residues of Grothendieck, or in
    terms of residue currents of Coleff-Herrera and Andersson-Wulcan,
    and we give several different proofs of non-degeneracy of the pairings.
    One of the proofs of non-degeneracy uses the theory of linkage, and
    conversely, we can use the non-degeneracy to obtain results about linkage
    for modules.
    We also discuss a variant of such pairings based on residues considered by
    Passare, Lejeune-Jalabert and Lundqvist.
\end{abstract}

\maketitle
\section{Introduction}

Let $\Ok = \Ok_{\C^n,0}$ be the local ring of germs of holomorphic functions at $0 \in \C^n$.
We will in this article mainly consider generalizations of the variant of the local duality theorem of Grothendieck,
as presented by Hartshorne in \cite{HLocal}*{Theorem~6.3}, see Theorem~\ref{thm:grothendieck-duality} below.
We will also later discuss a variant of this expressed as integrals, generalizing the variant of the local duality
theorem for Artinian complete intersection ideals as presented in \cite{GH}*{p.~693}.

An important purpose of this paper is to attempt to clarify the connection between the classical duality theory
in commutative algebra arising from the local cohomology theory of Grothendieck, and the theory of residue currents in complex analysis,
which has seen a strong development in recent years. More precisely, the variants of the local duality theorem of Grothendieck,
as referred to above, are various ``residue pairings'' whose existence and properties we show can be proved using both theories.
By constructing explicit maps we show that these seemingly different ways of defining the pairings indeed give the same pairings.
In the case of Artinian complete intersection ideals, it is classical that there are such relations between these theories,
while for more general ideals or modules, such relations have only to a small extent been explored. 

\medskip

Let $G$ be a finitely generated $\Ok$-module. We let
\begin{align*}
    & G_{(p)} := \{ g \in G \mid \codim (\supp g) \geq p \} \text{ and } \\ &G^{(p)}:=G_{(p)}/G_{(p+1)}.
\end{align*}
The torsion elements of $G$ are precisely the elements in $G_{(1)}$.
For any $G$, there is a natural pairing $G \times \Hom(G,\Ok) \to \Ok$, given by $(g,\varphi) \mapsto \varphi(g)$.
Since $\Ok$ is torsion-free, elements of $G_{(1)}$ are mapped to $0$, so the pairing descends to a pairing
\begin{equation} \label{eq:codim0pairing}
    G/G_{(1)} \times \Hom(G,\Ok) \to \Ok,
\end{equation}
and it is classical that this pairing is non-degenerate, see for example \cite{GR}*{3.3.3, p. 69}.

If $\codim G \geq 1$, then $G/G_{(1)} = 0$, so the pairing \eqref{eq:codim0pairing} is uninteresting,
but we consider here a variant of it.
For a germ of a subvariety $Z \subseteq (\C^n,0)$, we let $H^p_Z(\Ok)$ denote the $p$-th local cohomology
module of $\Ok$ with support in $Z$, see Section~\ref{ssect:grothendieck-preliminaries}.

\medskip

\begin{thm} \label{thm:local-duality}
    Let $G$ be a finitely generated $\Ok$-module of codimension $\geq p$,
    and let $Z \supseteq \supp G$, where $Z$ has pure codimension $p$.
    There exists a canonical explicitly given pairing
    \begin{equation} \label{eq:local-duality-pairing}
        G \times \Ext^p(G,\Ok) \to H^p_Z(\Ok),
    \end{equation}
    which is functorial in $G$, and which descends to a non-degenerate pairing
    \begin{equation} \label{eq:local-duality-pairing-descended}
        G/G_{(p+1)} \times \Ext^p(G,\Ok) \to H^p_Z(\Ok).
    \end{equation}
    The pairing \eqref{eq:local-duality-pairing} can be defined by either \eqref{eq:pairing-grothendieck}
    or \eqref{eq:pairing-aw}.
\end{thm}

When we say that the pairing is functorial in $G$, we mean that the following diagram
commutes for any finitely generated $\Ok$-module $F$ of codimension $\geq p$
and morphism $\alpha : F \to G$, where $Z \subseteq (\Cn,0)$ has pure codimension $p$
and $Z \supseteq (\supp F) \cup (\supp G)$.
\begin{equation} \label{eq:functorial}
    % Not sure what I can do to get better spacing automatically (the ---\dots--- around \times)
    \begin{gathered}
    \xymatrix{G      & *-------{\times} & \Ext^p(G,\Ok) \ar[r] \ar[d]_{\alpha^*} & H^p_Z(\Ok) \ar@{=}[d] \\
    F \ar[u]^{\alpha}& *-------{\times} & \Ext^p(F,\Ok) \ar[r] & H^p_Z(\Ok)}
    \end{gathered}
\end{equation}
%Note that by Lemma~\ref{lma:HpZW}, the choice of $Z$ is not really important.

In the case when $p = n$, Theorem~\ref{thm:local-duality} becomes a special case of the local
duality theorem of Grothendieck, \cite{HLocal}*{Theorem~6.3}, see Remark~\ref{rem:grothendieck-general}.
For general $p$, that there exists a pairing \eqref{eq:local-duality-pairing} is a straightforward
generalization of construction of the pairing in \cite{HLocal}.
However, the formulation of the fact that this descends to a non-degenerate pairing \eqref{eq:local-duality-pairing-descended}
is to our knowledge new, when $p < n$. In addition, even in the already known case, when $p = n$, our proof is
quite different from the proof in \cite{HLocal}*{Theorem~6.3}.

This theorem is also very close to the duality theorem of Andersson in \cite{AndNoeth}*{Theorem~1.2},
which deals with the case when $G$ has pure codimension $p$, although there, functoriality is not proven.
In \cite{AndNoeth}, this pairing is defined with the help of certain residue currents constructed
by Andersson and Wulcan, \cite{AW1}. It is also proven that this pairing coincides with the pairing from
\cite{HLocal} (without explicitly referring to it). We give a direct proof of functoriality for
this pairing using a comparison formula for residue currents from \cite{LarComp}.

We consider now the special case when $G = \Ok/I$, where $I$ is a complete intersection ideal of codimension $p$,
i.e., $I$ can be defined by $p$ functions $I = J(f_1,\dots,f_p)$.
In this case, one can define the pairing \eqref{eq:local-duality-pairing} with the help of residue currents of
Coleff and Herrera, \cite{CH}, and the non-degeneracy of the pairing then becomes the duality theorem
for Coleff-Herrera products, as proven independently by Passare, \cite{PMScand}, and Dickenstein-Sessa, \cite{DS1},
see Section~\ref{sect:ci} below.
For certain morphisms between such complete intersection modules, one obtains functoriality from
the transformation law for Coleff-Herrera products, \cite{DS1} or \cite{DS2}, see Example~\ref{ex:ci-morphism}
and \eqref{eq:transformation-law}.

\medskip

We give three different proofs of the non-degeneracy of the pairing in Theorem~\ref{thm:local-duality}.
One proof which is heavily based on homological algebra in Section~\ref{sect:general-pairing}, another
based on the theory of residue currents, similar to the proof in \cite{AndNoeth}, in Section~\ref{sect:aw-pairing},
and a proof using the theory of linkage in Section~\ref{sect:linkage}, although this last proof only
gives non-degeneracy in the first argument when $G$ is cyclic, i.e., of the form $\Ok/J$ for some ideal $J \subseteq \Ok$.

\bigskip

We now return to the pairing \eqref{eq:codim0pairing}. That it is non-degenerate in the first argument
can be reformulated as that the morphism
\begin{equation} \label{eq:codim0induced}
    G/G_{(1)} \to \Hom(\Hom(G,\Ok),\Ok)
\end{equation}
given by $g \mapsto (\varphi \mapsto \varphi(g))$ is injective. It is also well-understood when
\eqref{eq:codim0induced} is an isomorphism, which
is the case if and only if $G/G_{(1)}$ is $S_2$, see for example \cite{ST}*{Corollary~1.21}.
By $S_2$, we refer to the Serre $S_k$-conditions, see Section~\ref{ssect:Zp-Sk}.
Note also that the other morphism induced by \eqref{eq:codim0pairing}
is the morphism $\Hom(G,\Ok) \to \Hom(G/G_{(1)},\Ok)$ given by $\varphi \mapsto (g \mapsto \varphi(g))$,
which is easily seen to be an isomorphism.
Our next result is that this holds also for the pairing \eqref{eq:local-duality-pairing-descended}.

\begin{thm} \label{thm:local-duality-surjective}
    Let $G$ and $Z$ be as in Theorem~\ref{thm:local-duality}. The injection
    \begin{equation} \label{eq:local-duality-surjection1}
        G/G_{(p+1)} \to \Hom(\Ext^p(G,\Ok),H^p_Z(\Ok))
    \end{equation}
    induced by the non-degenerate pairing \eqref{eq:local-duality-pairing-descended} is an isomorphism
    if and only if $G/G_{(p+1)}$ is $S_2$, and the injection
    \begin{equation} \label{eq:local-duality-surjection2}
        \Ext^p(G,\Ok) \to \Hom(G/G_{(p+1)},H^p_Z(\Ok))
    \end{equation}
    induced by the non-degenerate pairing \eqref{eq:local-duality-pairing-descended} is always an isomorphism.
\end{thm}

If $\codim G = n$, then $G = G/G_{(n+1)}$ is Cohen-Macaulay, and hence also $S_2$, so both \eqref{eq:local-duality-surjection1}
and \eqref{eq:local-duality-surjection2} are isomorphisms, which follows from \cite{HLocal}.
In case when $G$ has pure codimension $p \leq n$, then \eqref{eq:local-duality-surjection2}
follows from \cite{AndNoeth}*{Theorem~1.5}.
The remaining cases of Theorem~\ref{thm:local-duality-surjective} when $\codim G < n$ are to our
knowledge new.

\bigskip

We finally also consider the case when $G$ does not necessarily have codimension $\geq p$.

\begin{thm} \label{thm:local-duality-general}
    Let $G$ be a finitely generated $\Ok$-module, and let $Z \supseteq \supp G$, where
    $Z$ has pure codimension $p$.
    There exists a canonical explicitly given pairing
    \begin{equation} \label{eq:general-duality-pairing}
        G_{(p)} \times \Ext^p(G,\Ok) \to H^p_Z(\Ok),
    \end{equation}
    which is functorial in $G$, and which descends to a non-degenerate pairing
    \begin{equation} \label{eq:general-duality-pairing-descended}
        G^{(p)} \times \Ext^p(G,\Ok)^{(p)} \to H^p_Z(\Ok).
    \end{equation}
    The injection
    \begin{equation} \label{eq:general-duality-surjection1}
        G^{(p)} \to \Hom(\Ext^p(G,\Ok)^{(p)},H^p_Z(\Ok))
    \end{equation}
    induced by the non-degenerate pairing \eqref{eq:general-duality-pairing-descended} is an isomorphism
    if and only if $G^{(p)}$ is $S_2$, and the injection
    \begin{equation} \label{eq:general-duality-surjection2}
        \Ext^p(G,\Ok)^{(p)} \to \Hom(G^{(p)},H^p_Z(\Ok))
    \end{equation}
    induced by the non-degenerate pairing \eqref{eq:general-duality-pairing-descended} is an isomorphism
    if and only if $\Ext^p(G,\Ok)^{(p)}$ is $S_2$.
    If $G$ has codimension $\geq p$, then \eqref{eq:local-duality-pairing} and \eqref{eq:general-duality-pairing} coincide.
    The pairing \eqref{eq:general-duality-pairing} can be defined by either \eqref{eq:pairing-grothendieck}
    or \eqref{eq:pairing-aw}.
\end{thm}

We note two things which are implied from the formulation of these theorems.
First of all, in \eqref{eq:local-duality-pairing-descended}, we have $\Ext^p(G,\Ok)$ in the second argument,
while in \eqref{eq:general-duality-pairing-descended}, we have $\Ext^p(G,\Ok)^{(p)}$ in the second argument.
Both these pairings are non-degenerate, and in the case when $G$ has codimension $\geq p$, these should coincide,
i.e., $\Ext^p(G,\Ok)$ must have pure codimension $p$ (or be $0$). This is indeed well-known, cf., Proposition~\ref{prop:ext-lc-pure}
below.
Secondly, if $G$ is still assumed to be of codimension $\geq p$, then \eqref{eq:local-duality-surjection2} and
\eqref{eq:general-duality-surjection2} coincide, and the first morphism is always an isomorphism, while the
second is an isomorphism if and only if $\Ext^p(G,\Ok)^{(p)} = \Ext^p(G,\Ok)$ is $S_2$. Thus, $\Ext^p(G,\Ok)$ is always $S_2$
when $G$ has codimension $\geq p$, see Corollary~\ref{cor:extisS2}.

We now consider a consequence of Theorem~\ref{thm:local-duality-general} and Theorem~\ref{thm:local-duality-surjective},
which can be seen as another generalization of \eqref{eq:codim0induced}.
We consider an arbitrary finitely generated $\Ok$-module $G$. Then, we have the injection \eqref{eq:general-duality-surjection1},
which is an isomorphism if and only if $G^{(p)}$ is $S_2$. In addition, $\Ext^p(G,\Ok)$ has codimension $\geq p$, see
Proposition~\ref{prop:ext-lc-pure}, so using \eqref{eq:local-duality-surjection2} (with $G$ replaced by $\Ext^p(G,\Ok)$),
we obtain the following corollary.

\begin{cor} \label{cor:roos}
    Let $G$ be a finitely generated $\Ok$-module. Then there is a natural injective map
    \begin{equation} \label{eq:roos}
        G^{(p)} \to \Ext^p(\Ext^p(G,\Ok),\Ok),
    \end{equation}
    which is surjective if and only if $G^{(p)}$ is $S_2$.
\end{cor}

This result is similar to the fundamental theorem of Roos, \cite{Roos}, for which the results and proof has
been elaborated by Björk, in for example \cite{BjRings}*{Chapter~2}. Part of the theorem of Roos states that
there exists an injective morphism \eqref{eq:roos}, and that the morphism is surjective outside a set of codimension $\geq p+2$.
The surjectiveness outside a set of codimension $\geq p+2$ is in fact also enough to prove that it is an isomorphism
if and only if it is $S_2$. Note that by Proposition~\ref{prop:ext-lc-pure} and Corollary~\ref{cor:extisS2}, $\Ext^p(\Ext^p(G,\Ok),\Ok)$ is $S_2$, and
by Corollary~\ref{cor:S2-isom}, the morphism is then surjective if and only if $G^{(p)}$ is $S_2$.
This result thus shows that $\Ext^p(\Ext^p(G,\Ok),\Ok)$ is the $S_2$-ification of $G^{(p)}$ in the sense that
it is a finite $\Ok$-module containing $G^{(p)}$ and which coincides with $G^{(p)}$ where $G^{(p)}$ is $S_2$.

In \cite{AndNoeth}, it is explained that a version from \cite{AndNoeth} of Theorem~\ref{thm:local-duality}
when $G = G^{(p)}$ has pure codimension $p$, follows from Corollary~\ref{cor:roos},
and also that the form of Corollary~\ref{cor:roos} from \cite{Roos} can be used to prove the version
of Theorem~\ref{thm:local-duality} from \cite{AndNoeth}.

\bigskip

We will now relate our results with a result from the theory of linkage. We first recall
that if $I$ and $J$ are two ideals in a ring $R$, then $I : J$ is the ideal
\begin{equation*}
    I : J = \{ r \in R \mid rJ \subseteq I \}.
\end{equation*}

Let $J \subseteq \Ok$ be an ideal of codimension $p$. The equidimensional hull $J_{[p]}$
of $J$ is the intersection of all primary components of $J$ of codimension $p$, cf., Remark~\ref{rem:equihull}.
The following result can be found in (the proof of) \cite{Vasc}*{Proposition~3.41}.
If $f = (f_1,\dots,f_p)$ is a regular sequence in $J$, generating the ideal $I$, then
\begin{equation} \label{eq:link-ideals}
    J_{[p]} = I : (I : J).
\end{equation}
This is a generalization of the case when $\Ok/J$ is Cohen-Macaulay from the
fundamental article \cite{PS} by Peskine and Szpiro.

If $G$ is a finitely generated $\Ok$-module, then $G$ is isomorphic to a quotient
module $G \cong \Ok^r/J$, for some submodule $J \subseteq \Ok^r$. Then, just as for an
ideal $J \subseteq \Ok$, one can consider the module
\begin{equation*}
    I : J = \{ g \in \Hom(\Ok^r,\Ok) \mid g(J) \subseteq I \},
\end{equation*}
which thus coincides with the usual colon ideal if $J \subseteq \Ok$. In addition, if $J \subseteq \Ok^r$
is a submodule such that $\Ok^r/J$ has codimension $p$, we let $J_{[p]}$ be the intersection of
all primary components of $J$ of codimension $p$.
We then obtain the following generalization of \eqref{eq:link-ideals}.

\begin{thm} \label{thm:colon-modules}
    Let $G$ be a finitely generated $\Ok$-module of codimension $p$,
    and write $G \cong \Ok^r/J$.
    If $f = (f_1,\dots,f_p)$ is a regular $J$-sequence, generating the ideal
    $I$, then
    \begin{equation} \label{eq:link-modules}
        J_{[p]} = I : (I : J).
    \end{equation}
\end{thm}

We note that $I : (I : J)$ is a submodule of $\Hom(\Hom(\Ok^r,\Ok),\Ok)$,
and that we have a natural isomorphism
\begin{equation} \label{eq:homhom}
    g \in \Ok^r \stackrel{\cong}{\mapsto}  (\varphi \mapsto \varphi(g)) \in \Hom(\Hom(\Ok^r,\Ok),\Ok).
\end{equation}
The module $J_{[p]}$ is a submodule of $\Ok^r$, and the equality \eqref{eq:link-modules} is
to be understood using this isomorphism \eqref{eq:homhom}.

We obtain Theorem~\ref{thm:colon-modules} as a rather easy consequence of non-degeneracy
in the first argument in the pairing in Theorem~\ref{thm:local-duality}. In fact, one
can also go the other way, and \eqref{eq:link-ideals} or \eqref{eq:link-modules} implies
non-degeneracy in the first argument in Theorem~\ref{thm:local-duality} in the corresponding
cases, cf., Lemma~\ref{lma:reduction-first}.

\subsection{A cohomological residue pairing} \label{ssect:intro-coh}

We now turn to a somewhat different formulation of the duality theorem.
If $G = \Ok/I$, where $I$ is a complete intersection ideal of codimension $n$, then
in \cite{GH}*{p. 693}, a pairing
\begin{equation} \label{eq:local-duality-pairing-residue-n}
    \Ok/I \times \Ext^n(\Ok/I,\Omega^n) \to \C
\end{equation}
is defined, which is non-degenerate in both arguments. Here, $\Omega^n$ is the module of germs of holomorphic $n$-forms near $0\in \C^n$.  We first recall how this is defined.
First of all, we represent $\Ext^n(\Ok/I,\Omega^n)$ as $H^n(\Hom(K_\bullet,\Omega^n))$,
where $(K,\psi) = (\Lambda \Ok^n,\delta_f)$ is the Koszul complex of $f$, and an element
$\omega \in H^n(\Hom(K_\bullet,\Omega^n))$ can thus be represented as
$[\omega_0 (e_1 \wedge \dots \wedge e_n)^*]$, where $\omega_0$ is a holomorphic $(n,0)$-form,
see Section~\ref{sect:ci} for notation.
The pairing \eqref{eq:local-duality-pairing-residue-n} is then defined by
\begin{equation} \label{eq:artinian-ci-explicit}
    (g, \omega) \mapsto \frac{1}{(2\pi i)^n} \int_{\{ \cap_{i=1}^n |f_i|=\epsilon \}} \frac{g \omega_0}{f_1 \dots f_n},
\end{equation}
where $\epsilon$ is chosen small enough such that $f_1,\dots,f_n$ and $\omega_0$ are holomorphic on
$D_\epsilon := \{ \cap_{i=1}^n |f_i|\leq\epsilon \}$ and $(f_1,\dots,f_n)$ has an isolated common zero at $\{ 0 \}$ in
$D_\epsilon$, cf., \cite{GH}*{Chapter~5.1}. This pairing is canonical, and does not
depend on the choice of generators $f$ of $I$.

Passare constructed in \cite{PMScand} a generalization of \eqref{eq:artinian-ci-explicit},
to the case of complete intersection ideals of arbitrary codimension, although the viewpoint
of it as a canonical pairing similar to \eqref{eq:local-duality-pairing-residue-n} was not elaborated.
In order to describe this construction, we let $Z \subseteq (\C^n,0)$ be a subvariety of pure
codimension $p$, and we then define $H^{p,q}_{Z^c}$ to be the module of germs at $0$ of smooth
$(p,q)$-forms with compact support, such that they are $\dbar$-closed in a neighbourhood of $Z$.

By \cite{GH}*{p. 651--655} and Stokes' theorem, there exists a $(0,n-1)$-form $B_f$
(being essentially like the pullback of the Bochner-Martinelli kernel by $f$), such that if $\chi$ is
a cut-off function which is $\equiv 1$ in a neighbourhood of $\|f\|=\epsilon$, then the right-hand side of
\eqref{eq:artinian-ci-explicit} equals
\begin{equation*}
    \int g \omega_0 \wedge B_f \wedge \dbar \chi.
\end{equation*}
More generally, one then gets a pairing
\begin{equation*}
    \Ok/I \times \Ext^n(\Ok/I,\Ok) \to \Hom_\C(H^{n,0}_{\{0\}^c},\C),
\end{equation*}
which if $\beta \in H^{n,0}_{\{0\}^c}$ and $\xi = [\xi_0 e^*] \in \Ext^n(\Ok/I,\Ok)$,
then the pairing is given by
\begin{equation} \label{eq:pairing-pa}
    \la g, \xi \ra_{GH}(\beta) := \int g \xi_0 B_f \wedge \dbar \beta,
\end{equation}
and by taking $\beta = \chi dz_1 \wedge \dots \wedge dz_n$, one obtains that the
pairing is non-degenerate.

In \cite{PMScand}, Passare then showed that for $I = J(f_1,...,f_p)$ a complete intersection
ideal of codimension $p \leq n$,
one can define a $(0,p-1)$-form $B_f$ similar to above, such that if $g \in \Ok$, then
\begin{equation} \label{eq:passare-duality}
    g \in I \text{ if and only if } \int g  B_f \wedge \dbar \beta = 0 \text{ for all $\beta \in H^{n,n-p}_{Z(I)^c}$}.
\end{equation}
With the help of \eqref{eq:passare-duality}, and using a similar construction as in \cite{GH}, one then obtains a non-degenerate
pairing
\begin{equation} \label{eq:passare-pairing}
    \Ok/I \times \Ext^p(\Ok/I,\Ok) \to \Hom_\C(H^{n,n-p}_{Z(I)^c},\C).
\end{equation}
Inspired by the construction of residue currents of Andersson and Wulcan, Lundqvist generalized in \cite{Lund2}
\eqref{eq:passare-duality} to pure dimension ideals, and one can also use this construction to obtain a
pairing \eqref{eq:passare-pairing}, which is described in Section~\ref{ssect:lund-def}.

We now compare these pairings with the pairing from Theorem~\ref{thm:local-duality} and Theorem~\ref{thm:local-duality-general}.
A local cohomology class $T \in H^p_Z(\Ok)$ can be represented by a $\dbar$-closed $(0,p)$-current  $\mu$ with support on $Z$,
modulo $\dbar$ of $(0,p-1)$-currents with support on $Z$, see \eqref{eq:lc-current-repr}. One then obtains a map
\begin{equation} \label{eq:local-cohomology-dual}
    R : H^p_Z(\Ok) \to \Hom_\C(H^{n,n-p}_{Z^c},\C),
\end{equation}
where $R(T) \in \Hom_\C(H^{n,n-p}_{Z^c},\C)$ is given as $R(T)(\gamma) = \int \mu \wedge \gamma$.
This exists since $\mu \wedge \gamma$ is a $(n,n)$-current with compact support,
and by the fact that $\gamma$ is $\dbar$-closed near $Z$, it follows easily that $R(T)$ so
defined is independent of the choice of representative $\mu$ of $T$.

\begin{thm} \label{thm:cohomological-duality}
    Let $G$ be a finitely generated $\Ok$-module, and let $Z \supseteq \supp G_{(p)}$, where $Z$ has pure codimension $p$.
    There exists a non-degenerate canonical explicitly given pairing
    \begin{equation} \label{eq:cohomological-duality-pairing}
        G^{(p)} \times \Ext^p(G,\Ok)^{(p)} \to \Hom_\C(H^{n,n-p}_{Z^c},\C),
    \end{equation}
    which is functorial in $G$.
    The pairing \eqref{eq:cohomological-duality-pairing} can be defined by composing the pairing
    \eqref{eq:general-duality-pairing-descended} with \eqref{eq:local-cohomology-dual}, and if
    $G$ has codimension $\geq p$, then this pairing coincides with \eqref{eq:pairing-lu}.
\end{thm}

Since the pairing
in \eqref{eq:cohomological-duality-pairing} can be defined in terms of the pairing
\eqref{eq:general-duality-pairing-descended}, non-degeneracy in Theorem~\ref{thm:cohomological-duality}
implies non-degeneracy in Theorem~\ref{thm:local-duality-general}.
The pairing \eqref{eq:pairing-lu} is the pairing defined by Lundqvist. By the main result of \cite{Lund2},
if $G = \Ok/J$ has pure codimension $p$, then the pairing \eqref{eq:pairing-lu} is non-degenerate in the
first argument, which thus implies the non-degeneracy in the first argument
in Theorem~\ref{thm:local-duality} in this case. In fact, one can also go the other way around,
and non-degeneracy in Theorem~\ref{thm:local-duality-general} implies non-degeneracy also
in Theorem~\ref{thm:cohomological-duality}, see Lemma~\ref{lma:lc-injective}.

We also mention that a variant of such residues was considered also by Lejeune-Jalabert,
see \cite{LJ1} and \cite{LJ2}, although the main purpose of these articles was to obtain
explicit representations of the fundamental cycle of Cohen-Macaulay ideals.
Especially in the Artinian case, she obtained explicit expressions for such residues as integrals,
which we by functoriality of the pairing of Lundqvist (Proposition~\ref{prop:lu-functorial}),
can see that these definitions indeed coincide, see Example~\ref{ex:lj}.
In \cite{LJ1} and \cite{LJ2}, non-degeneracy of the pairing was not considered.

\subsection{Structure of the proofs of Theorem~\ref{thm:local-duality}, \ref{thm:local-duality-surjective} and \ref{thm:local-duality-general}}

The description of the pairings in Theorem~\ref{thm:local-duality} and \ref{thm:local-duality-general},
and the proofs of these theorems and Theorem~\ref{thm:local-duality-surjective} occupy the majority of the article,
and is divided into several parts. In addition, for some of the statements, we give several different proofs,
so we briefly outline here the disposition of these proofs.

First of all, we give two different ways of defining the pairings \eqref{eq:local-duality-pairing}
and \eqref{eq:general-duality-pairing}, which are given either
algebraically by \eqref{eq:pairing-grothendieck}, or analytically by \eqref{eq:pairing-aw}.
That these pairings are both functorial is proven in Lemma~\ref{lma:functorial-grothendieck}
and Lemma~\ref{lma:functorial-aw}. That these pairings then descend to the pairings
\eqref{eq:local-duality-pairing-descended} and \eqref{eq:general-duality-pairing-descended} then
follows by a general result about pairings of this form, Corollary~\ref{cor:descends}.
That these two pairings coincide is proven in Proposition~\ref{prop:aw-gr-coincide}.

For the pairing \eqref{eq:pairing-grothendieck}, the remaining parts about non-degeneracy, 
in Theorem~\ref{thm:local-duality} as well as the statements about surjectivity of the induced morphisms
in Theorem~\ref{thm:local-duality-surjective} follow from Lemma~\ref{lma:reduction-to-ext-pure} and
Proposition~\ref{prop:pure-nondeg-surjective}.
Finally, the remaining parts of Theorem~\ref{thm:local-duality-general} are proven in
Proposition~\ref{prop:reduction-codim-p}.

In addition, for the pairing \eqref{eq:pairing-aw}, we prove non-degeneracy in Proposition~\ref{prop:nondegenerate-aw}.
We also give an alternative proof of non-degeneracy of the pairing \eqref{eq:pairing-grothendieck} for $G = \Ok/J$,
where $J$ has codimension $\geq p$ in Section~\ref{ssect:ci-reduction}, see Lemma~\ref{lma:reduction-second}
and Lemma~\ref{lma:reduction-first}.

\section*{Acknowledgements}

I would like to thank Mats Andersson, H{\aa}kan Samuelsson and Elizabeth Wulcan
for valuable discussions in the preparation of this article,
Gerhard Pfister for pointing me to a reference for a result I needed,
and anonymous referees of the present and an earlier version of the article for useful suggestions
to improve the article.

\section{The Grothendieck pairing}

We begin by recalling the definition of local cohomology, and the statement of the local
duality theorem from \cite{HLocal}, in order to compare with our theorems, and as well describe
the first of the definitions of the pairings \eqref{eq:local-duality-pairing}
and \eqref{eq:general-duality-pairing}.

\subsection{Local cohomology and the Grothendieck pairing} \label{ssect:grothendieck-preliminaries}

If $R$ is a Noetherian commutative ring, $J \subseteq R$ an ideal, and $G$ an $R$-module,
we define the \emph{$j$-th local cohomology module} $H^j_J(G)$ of $G$ with support in $J$ as
\begin{equation*}
    H^j_J(G) := \varinjlim_t \Ext^j(\Ok/J^t,G).
\end{equation*}
A rather extensive treatment of local cohomology and applications can be found in \cite{24Loc}.
We remark that the local cohomology modules are not the same as the germs of the local cohomology
sheaves, see Remark~\ref{rem:moderate-cohomology}.

As the local cohomology modules $H^k_J(G)$ only depend on the radical of $J$,
see \cite{24Loc}*{Proposition~7.3 and Theorem~7.8},
if $Z \subseteq (\Cn,0)$ is a germ of an analytic subvariety of $\Cn$,
and $J \subseteq \Ok$ is an ideal such that $Z = Z(J)$, then we write the corresponding local cohomology as
\begin{equation} \label{eq:localcohomology}
    H^j_Z(G) := H^j_J(G) = \varinjlim_t \Ext^j(\Ok/J^t,G).
\end{equation}
If $\supp G \subseteq Z$, then for $t \gg 1$, $J^t \subseteq \ann G$, so $H^0_Z(G) \cong G$.
If $Z$ has codimension $p$, and $\supp G_{(p)} \subseteq Z$, then
then $H^0_Z(G) = G_{(p)}$.

Inspired by Serre's duality theorem for smooth projective varieties, Grothendieck
obtained the following local duality theorem, formulated in terms of local cohomology,
\cite{HLocal}*{Theorem~6.3}. To compare this result with Theorem~\ref{thm:local-duality-general}, Theorem~\ref{thm:local-duality} and Theorem~\ref{thm:local-duality-surjective}),
in these theorems we allow local cohomology with support in more general varieties than just $Z = \{ 0 \}$,
but on the other hand, we only consider the ring $R = \Ok$, and only the case $i = 0$.

\begin{thm} \label{thm:grothendieck-duality}
    Let $R$ be a Gorenstein local ring, $G$ a finitely generated $R$-module, and $\mathfrak{m}$ the maximal ideal in $R$.
    Then there is a natural pairing
 \begin{equation} \label{eq:pairing}
    H^i_\mathfrak{m}(G) \times \Ext^{n-i}(G,R) \to H^n_\mathfrak{m}(R),
 \end{equation}
 inducing isomorphisms
 \begin{align}
     \label{eq:pairing-isomorphism} H^i_\mathfrak{m}(G) \cong \Hom(\Ext^{n-i}(G,R),H^n_\mathfrak{m}(R)) \text{ and } \\
    \label{eq:pairing-isomorphism-2} \Ext^{n-i}(G,R)^{\wedge} \cong \Hom(H^i_\mathfrak{m}(G),H^n_\mathfrak{m}(R)),
 \end{align}
 where $^\wedge$ denotes completion with respect to the $\mathfrak{m}$-adic topology.
\end{thm}

\begin{remark} \label{rem:grothendieck-general}
In case $i=0$ and $p=n$, then $G^{(n)}=G_{(n)} = H^0_\mathfrak{m}(G)$, so when $R=\Ok$, then \eqref{eq:general-duality-pairing-descended} and \eqref{eq:pairing} are pairings of the same modules.
The morphisms \eqref{eq:general-duality-surjection1}, \eqref{eq:general-duality-surjection2}, \eqref{eq:pairing-isomorphism} and \eqref{eq:pairing-isomorphism-2} induced by these pairings a priori seem to have different properties.
First of all, one takes the completion in \eqref{eq:pairing-isomorphism-2}, but since $\Ext^n(G,\Ok)$ is Artinian,
see Proposition~\ref{prop:ext-lc-pure} below, its completion with respect to $\mathfrak{m}$ is the module itself.
Secondly, since $G^{(n)}$ and $\Ext^n(G,\Ok)$ are Artinian, they are Cohen-Macaulay,
and hence \eqref{eq:general-duality-surjection1} and \eqref{eq:general-duality-surjection2} are both isomorphisms,
and then coincide with the isomorphisms \eqref{eq:pairing-isomorphism} and \eqref{eq:pairing-isomorphism-2},
so Theorem~\ref{thm:grothendieck-duality} and Theorem~\ref{thm:local-duality-general} coincide when
$R = \Ok$, $i = 0$ and $p = n$.
\end{remark}

We elaborate here a bit on the definition of the pairing in \eqref{eq:pairing}.
The definition of the pairing \eqref{eq:pairing} as described in \cite{HLocal}*{Chapter~6}
works equally well more generally to give a pairing
\begin{equation} \label{eq:pairing-general}
    H^i_Z(G) \times \Ext^{p-i}(G,\Ok) \to H^p_Z(\Ok).
\end{equation}
The pairing \eqref{eq:pairing-general} is defined in terms of the so-called Yoneda
pairing of $\Ext$, which is a pairing of the form
\begin{equation*}
    \Ext^i(A,B) \times \Ext^j(B,C) \to \Ext^{i+j}(A,C),
\end{equation*}
and which is described in \cite{HLocal}*{Chapter~6.1}.
Then, the pairing \eqref{eq:pairing-general} is defined as follows. First of all, an element
$g \in H^i_Z(G)$ can be represented as an element $g_0 \in \Ext^i(\Ok/J^t,G)$ by \eqref{eq:localcohomology}.
Taking the Yoneda pairing with an element $\xi \in \Ext^{p-i}(G,\Ok)$, we obtain an element
$(g_0,\xi)_Y \in \Ext^p(\Ok/J^t,\Ok)$. The desired element is then obtained by composing with the
map
\begin{equation} \label{eq:pi-t}
    \pi_t : \Ext^p(\Ok/J^t,\Ok) \to \varinjlim_s \Ext^p(\Ok/J^s,\Ok) \cong H^p_Z(\Ok).
\end{equation}
Using the notation from above, the pairing \eqref{eq:pairing-general} is defined
as
\begin{equation} \label{eq:pairing-general-definition}
    \la g, \xi \ra := \pi_t (g_0,\xi)_Y.
\end{equation}

\subsection{The associated primes of $\Ext^p(G,\Ok)$}

In order to obtain the results that the pairings in Theorem~\ref{thm:local-duality}
and Theorem~\ref{thm:local-duality-general} descend, we will use the following
result about the codimension and associated primes  of $\Ext$-groups,
which follows from \cite{EHV}*{Theorem~1.1}.

\begin{prop} \label{prop:ext-ass-primes}
    Let $G$ be a finitely generated $\Ok$-module. Then, $\Ext^p(G,\Ok)$ has codimension $\geq p$, and
    the associated primes of codimension $p$ of $G$ and $\Ext^p(G,\Ok)$ coincide.
\end{prop}

In particular, if $G$ has no associated primes of codimension $k$, then
\begin{equation} \label{eq:pure-ass-primes}
    \codim \Ext^k(G,\Ok) \geq k+1.
\end{equation}

When $p = \codim G$, we have the following information about $\Ext^p(G,\Ok)$, see \cite{BjRings}*{Lemma~7.11},
cf., also the discussion after \cite{HL}*{Lemma~1.1.8}.

\begin{prop} \label{prop:ext-kp}
    Let $G$ be a finitely generated $\Ok$-module, and let $p = \codim G$.
    Then 
    \begin{equation} \label{eq:ext-kp}
        \codim \Ext^k(\Ext^p(G,\Ok),\Ok) \geq k+2 \text{ for } k > p.
    \end{equation}
\end{prop}

We remark that in \cite{BjRings}, the result is stated in terms of $\depth M$, the depth of a 
module $M$, which in our case equals $\codim M$, since $\Ok$ is Cohen-Macaulay.

\begin{prop} \label{prop:ext-lc-pure}
    \noindent a) Let $G$ be a finitely generated $\Ok$-module of codimension $p$.
    Then $\Ext^p(G,\Ok)$ has pure codimension $p$.

    \noindent b) Let $Z \subseteq (\C^n,0)$ be a subvariety of pure codimension $p$.
    Then $H^p_Z(\Ok)$ has pure codimension $p$.
\end{prop}

\begin{proof}
    Part a) follows by combining Proposition~\ref{prop:ext-ass-primes} and Proposition~\ref{prop:ext-kp},
    since if $\Ext^p(G,\Ok)$ has an associated prime of codimension $k > p$, then we would have
    that $\codim \Ext^k(\Ext^p(G,\Ok),\Ok) = k$.

    Part b) is then a consequence of a) as follows:
    If $P$ is an associated prime of $H^p_Z(\Ok)$, then there exists some $\mu \in H^p_Z(\Ok)$
    such that $\ann \mu = P$. By \eqref{eq:localcohomology}, if $t \gg 1$, there exists some $\mu_t \in \Ext^p(\Ok/J^t,\Ok)$
    representing $\mu$, where $J$ is an ideal such that $Z(J) = Z$. By the Noetherianness of $\Ok$, we can assume that
    $t \gg 1$ is such that $\ann \mu_t = P$. Hence, by a), $\codim P = p$.
\end{proof}

For Proposition~\ref{prop:ext-lc-pure}, it would have sufficed that the right-hand side of \eqref{eq:ext-kp} was $k+1$ instead of $k+2$,
but in the proof of Theorem~\ref{thm:local-duality-surjective} and so on, in Section~\ref{sect:general-pairing},
we will need the stronger inequality with $k+2$.

\begin{cor} \label{cor:descends}
    Assume that $A$ and $B$ are $\Ok$-modules of codimension $\geq p$, and let $Z$ be a subvariety
    of pure codimension $p$.
    Then any $\Ok$-bilinear pairing
    \begin{equation*}
        A \times B \to H^p_Z(\Ok)
    \end{equation*}
    descends to a pairing
    \begin{equation*}
        A/A_{(p+1)} \times B/B_{(p+1)} \to H^p_Z(\Ok).
    \end{equation*}
\end{cor}

\begin{proof}
    Clearly, for any $\Ok$-bilinear pairing, $\supp \la a,b \ra \subseteq (\supp a) \cap (\supp b)$.
    Thus, if the support of either $a$ or $b$ has codimension $\geq p+1$, then
    $\la a,b \ra$ has codimension $\geq p+1$, and thus is $0$, since $H^p_Z(\Ok)$ has pure codimension $p$ by Proposition~\ref{prop:ext-lc-pure}.
\end{proof}

As a consequence, any way of defining the pairings \eqref{eq:local-duality-pairing}
and \eqref{eq:general-duality-pairing} will descend to pairings \eqref{eq:local-duality-pairing-descended}
and \eqref{eq:general-duality-pairing-descended}.

\subsection{The comparison morphism} \label{ssect:comparison-morphism}

We will need the following result about how a morphism of modules induces a
morphism of free resolutions of the modules.

\begin{prop} \label{prop:comparison-morphism}
    Let $\alpha : F \to G$ be a homomorphism of finitely generated $\Ok$-modules, and let $(K,\psi)$ and
    $(E,\varphi)$ be free resolutions of $F$ and $G$.
    Then, there exists a morphism $a : (K,\psi) \to (E,\varphi)$ of complexes which extends $\alpha$.

    If $b$ is any other such morphism, then there exists a homotopy $s : (K,\psi) \to (E,\varphi)$
    of degree $-1$, i.e., consisting of morphisms $s_k : K_k \to E_{k+1}$,
    such that $a_i - b_i = \varphi_{i+1} s_i - s_{i-1} \psi_i$.
\end{prop}

We say that $a$ \emph{extends} $\alpha$ if the map induced by $a_0$ on
$K_0/(\im \psi_1) \cong F \to G \cong E_0/(\im \varphi_1)$ equals $\alpha$.
Both the existence and uniqueness up to homotopy of $a$ follow from defining $a$ or
$s$ inductively by a relatively straightforward diagram chase,
see \cite{Eis}*{Proposition~A3.13}.
In Example~\ref{ex:ci-morphism} below, we give examples of when such a morphism can be explicitly
constructed for certain morphisms between Koszul complexes.

Note in particular, that if one represents $\Ext^p(F,\Ok)$ and $\Ext^p(G,\Ok)$ as $H^p(\Hom(K_\bullet,\Ok))$
and $H^p(\Hom(E_\bullet,\Ok))$, then $\alpha^* : \Ext^p(F,\Ok) \to \Ext^p(G,\Ok)$ is given by $a_p^*$.

\subsection{Definition and properties of the pairing} \label{ssect:grothendieck-definition}

In this section, we give the first explicit way of defining the pairings \eqref{eq:local-duality-pairing}
and \eqref{eq:general-duality-pairing}.
This definition of the pairing coincides with the pairing \eqref{eq:pairing-general}, but since
we only consider a special case, i.e., when $i = 0$, we can describe the pairing more concretely.
Let $Z$ be such that $\supp G_{(p)} \subseteq Z$, and let $J$ be an ideal such that $Z(J) = Z$.
Note that since $\supp G_{(p)} \subseteq Z$, by the Nullstellensatz,
if $g \in G_{(p)}$, then $J^t g = 0$ for $t \gg 1$. Thus, any element $g \in G_{(p)}$ defines
a morphism
\begin{equation} \label{eq:epsilon-g-def}
    \epsilon_g : \Ok/J^t \to G \text{ such that } \epsilon_g(1) = g.
\end{equation}
We thus get an induced morphism $\epsilon_g^* : \Ext^p(G,\Ok) \to \Ext^p(\Ok/J^t,\Ok)$.
We then compose this with the morphism \eqref{eq:pi-t}.
Using the notation from above, the \emph{Grothendieck pairing}
\begin{equation*}
    G_{(p)} \times \Ext^p(G,\Ok) \to H^p_Z(\Ok)
\end{equation*}
is defined as
\begin{equation} \label{eq:pairing-grothendieck}
    \la g, \xi \ra_{Gr} := \pi_t( \epsilon_g^* \xi ).
\end{equation}

\begin{lma} \label{lma:functorial-grothendieck}
    The Grothendieck pairing \eqref{eq:pairing-grothendieck} is functorial in $G$.
\end{lma}

\begin{proof}
    This follows directly from the functoriality of $\Ext^p(\bullet,\Ok)$, since if
    $f = \alpha(g)$, then $\epsilon_f = \alpha \epsilon_g$, and thus, $\epsilon_f^* = \epsilon_g^* \alpha^*$.
\end{proof}

\section{Complete intersection ideals} \label{sect:ci}

In this section, we consider the (already well-known) case of Theorem~\ref{thm:local-duality}, when
$G = \Ok/J$, where $J$ is a complete intersection ideal of codimension $p$.
In Theorem~\ref{thm:local-duality}, we also take $Z = Z(J)$, and take $J(g_1,\dots,g_p)$
as the defining ideal of $Z$. In this case, $J G = 0$, so for $g \in \Ok/J$, we can take $t = 1$ in defining the morphism
\eqref{eq:epsilon-g-def}, i.e., $\epsilon_g : \Ok/J \to \Ok/J$ is just multiplication with $g$.

Note that when $\epsilon_g : \Ok/J \to \Ok/J$ is just multiplication with $g$, then the induced pairing
$\epsilon_g^* : \Ext^p(\Ok/J,\Ok) \to \Ext^p(\Ok/J,\Ok)$ can be taken as just multiplication with $g$.
Thus, the pairing in \eqref{eq:pairing-grothendieck} is given by
\begin{equation} \label{eq:pairing-grothendieck-ci}
    \la g, \xi \ra = \pi_1(g\xi),
\end{equation}
where $\pi_1 : \Ext^p(\Ok/J,\Ok) \to H^p_Z(\Ok)$.

In order to prove non-degeneracy of the pairing, we use the following result about $\Ext^p(\Ok/J,\Ok)$,
which follows from \cite{GH}*{Proposition, p. 690}.
\begin{lma} \label{lma:ext-ci}
    Let $J$ be a complete intersection ideal of codimension $p$. Then there is a (non-canonical) isomorphism
    \begin{equation} \label{eq:ext-ci-isom}
        \Ok/J \to \Ext^p(\Ok/J,\Ok),
    \end{equation}
    and if $I \subseteq J$ is also a complete intersection ideal of codimension $p$,
    then the morphism
    \begin{equation*}
        \Ext^p(\Ok/J,\Ok) \to \Ext^p(\Ok/I,\Ok)
    \end{equation*}
    induced by the natural surjection $\Ok/I \to \Ok/J$ is injective.
\end{lma}

For future reference, we also make the isomorphism \eqref{eq:ext-ci-isom} explicit.
Since $J$ is a complete intersection ideal, the Koszul complex $(\bigwedge \Ok^p,\delta_g)$ of $g$,
which we will denote $(K,\psi)$, is a free resolution of $\Ok/J$, where we denote $e_1,\dots,e_p$
the standard basis for $\Ok^p$, such that $\delta_g$ is the contraction with $\sum g_i e_i^*$.
We can thus represent $\Ext^p(\Ok/J,\Ok)$ as $H^p(\Hom(K_\bullet,\Ok))$.
The element $e_1\wedge \dots \wedge e_p$ is a basis of $K_p$, and the isomorphism
\eqref{eq:ext-ci-isom} using this representation of $\Ext$ is given by
\begin{equation} \label{eq:ext-ci-isom-explicit}
    h \mapsto h (e_1 \wedge \dots \wedge e_p)^*.
\end{equation}

\begin{lma} \label{lma:ci-nondeg}
    Let $G = \Ok/J$, where $J$ is a complete intersection ideal of codimension $p$, and let $Z = Z(J)$.
    Then the pairing \eqref{eq:local-duality-pairing},
    \begin{equation*}
        \Ok/J \times \Ext^p(\Ok/J,\Ok) \to H^p_Z(\Ok),
    \end{equation*}
    as given by \eqref{eq:pairing-grothendieck}, is non-degenerate in both arguments.
\end{lma}

\begin{proof}
    We claim that $\pi_1 : \Ext^p(\Ok/J,\Ok) \to H^p_Z(\Ok)$ is injective, and thus,
    it is enough to prove non-degeneracy of the pairing
    \begin{equation*}
        \Ok/J \times \Ext^p(\Ok/J,\Ok) \to \Ext^p(\Ok/J,\Ok),
    \end{equation*}
    which is given by $(g,\xi) \mapsto g\xi$. The fact that this pairing is non-degenerate follows
    easily from the isomorphism \eqref{eq:ext-ci-isom}.

    One way of proving the claim is that in order
    to define $H^p_Z(\Ok)$, if $J_t$ is a family of ideals such that
    for any $t$, there exists $s$ and $r$ such that
    $J^s \subseteq J_r \subseteq J^t$, then
    \begin{equation} \label{eq:loc-reduced}
        H^p_Z(\Ok) \cong \varinjlim_t \Ext^p(\Ok/J_t,\Ok),
    \end{equation}
    see \cite{24Loc}*{Remark~7.9}.
    If we let $J_t := J(g_1^t,\dots,g_p^t)$, then $J_t \subseteq J^t$. In addition, by the pigeonhole
    principle, $J^{pt} \subseteq J_t$. Thus, we can represent $H^p_Z(\Ok)$ using \eqref{eq:loc-reduced}.
    Since $J_t \subseteq J$ is a complete intersection ideal of codimension $p$, the induced map
    \begin{equation*}
        \Ext^p(\Ok/J,\Ok) \to \Ext^p(\Ok/J_t,\Ok)
    \end{equation*}
    is injective by Lemma~\ref{lma:ext-ci}, so $\pi_1 : \Ext^p(\Ok/J,\Ok) \to H^p_Z(\Ok)$ is injective, proving
    the claim.

    Another less direct way to prove the claim is to instead use Lemma~\ref{lma:hom-to-extI} below.
\end{proof}

\subsection{Coleff-Herrera products}

If $f \in \Ok$, the \emph{principal value current} $1/f$,
can be defined by
\begin{equation*}
    \frac{1}{f} := \lim_{\epsilon \to 0^+} \frac{\bar{f}}{|f|^2 + \epsilon},
\end{equation*}
and satisfies $f(1/f) = 1$.
Using regularity of the $\dbar$-operator on distributions, it is then easily seen that $\ann_\Ok \dbar (1/f) = J(f)$,
i.e., $g \in \Ok$ lies in the annihilator of $\dbar(1/f)$, i.e.,  $g \dbar (1/f) = 0$, if and only if $g$ belongs to
the principal ideal $J(f)$ generated by $f$.
If $f = (f_1,\dots,f_p)$ is a tuple of holomorphic functions defining a complete intersection ideal of codimension $p$,
then Coleff and Herrera showed in \cite{CH} that one can give a reasonable meaning
to products of residue currents $\dbar(1/f_i)$, nowadays called the \emph{Coleff-Herrera product} of $f$, and written
\begin{equation*}
    \dbar \frac{1}{f_p} \wedge \dots \wedge \dbar \frac{1}{f_1}.
\end{equation*}
The \emph{duality theorem for Coleff-Herrera products}, proven independently by Passare, \cite{PMScand}, and Dickenstein-Sessa, \cite{DS1},
says that if $f$ defines a complete intersection ideal of codimension $p$, then
\begin{equation} \label{eq:ch-duality}
    \ann_\Ok \dbar \frac{1}{f_p} \wedge \dots \wedge \dbar \frac{1}{f_1} = J(f_1,\dots,f_p).
\end{equation}

\subsection{Representations of local cohomology classes as currents}

If $J \subseteq \Ok$ is an ideal, $\Ext^p(\Ok/J,\Ok)$ can be represented as $H^p(\Hom(\Ok/J,L_\bullet))$,
where $L$ is an injective resolution of $\Ok$.
Since the Dolbeault complex $(C^{0,\bullet},\dbar)$ of $(0,*)$-currents is an injective resolution of $\Ok$,
we can thus represent objects in $\Ext^p(\Ok/J,\Ok)$ as $\dbar$-closed $(0,p)$-currents annihilated by $J$.

In addition, if one represents $\Ext^p(\Ok/J^t,\Ok)$ as $H^p(\Hom(\Ok/J^t,C^{0,\bullet}))$, then the
morphism $\Ext^p(\Ok/J,\Ok) \to \Ext^p(\Ok/J^t,\Ok)$ is induced by the inclusion $\Hom(\Ok/J,C^{0,p}) \to \Hom(\Ok/J^t,C^{0,p})$,
which just corresponds to the fact that currents annihilated by $J$ are also annihilated by $J^t$.
Thus, using this representation, any element in $H^p_Z(\Ok)$
can be represented by a $(0,p)$-current annihilated by $J^t$ for $t \gg 1$, which due to the fact that a current has locally
finite order is equivalent to that it has support on $Z = Z(J)$. Thus, one has the following representation
of the local cohomology groups,
\begin{equation} \label{eq:lc-current-repr}
    H^p_Z(\Ok) \cong H^p(C^{0,\bullet}_Z),
\end{equation}
where $(C^{0,\bullet}_Z,\dbar)$ is the Dolbeault complex of $(0,*)$-currents with support on $Z$.

\begin{remark} \label{rem:moderate-cohomology}
The local cohomology groups we consider are in the local ring $\Ok = \Ok_{\C^n,0}$, which correspond
to the stalks of the \emph{moderate cohomology sheaf}, which is what is mainly treated in for example \cite{DS1}.
We remark however that these stalks are not the same as the stalks of the \emph{local cohomology sheaf},
cf., the introduction of \cite{DS1}.
\end{remark}

Another way of representing $\Ext^p(\Ok/J,\Ok)$ is as elements in $H^p(\Hom(E_\bullet,\Ok))$,
where $(E,\varphi)$ is a free resolution of $\Ok/J$, and by standard homological algebra,
there is a canonical isomorphism
\begin{equation} \label{eq:extrepr}
H^p(\Hom(E_\bullet,\Ok)) \cong H^p(\Hom(\Ok/J,L_\bullet)).
\end{equation}
If we in particular consider the case when $J$ is a complete intersection ideal of codimension $p$ as above,
generated by $f_1,\dots,f_p$, then one can take the Koszul complex $(K,\psi)$ of $f$ as a free resolution of $\Ok/J$,
and one has the representation \eqref{eq:ext-ci-isom-explicit} of $\Ext^p(\Ok/J,\Ok)$.
One thus gets a canonical isomorphism
\begin{equation} \label{eq:extrepr2}
    H^p(\Hom(K_\bullet,\Ok)) \cong H^p(\Hom(\Ok/J,C^{0,\bullet})).
\end{equation}
In \cite{DS1}, the canonical isomorphism \eqref{eq:extrepr2} is expressed in terms of the Coleff-Herrera product,
and is given by
\begin{equation} \label{eq:extrepr3}
    [\xi_0] = [h (e_1\wedge \dots \wedge e_p)^*] \mapsto \left[h \dbar \frac{1}{f_p}\wedge \dots \wedge \dbar \frac{1}{f_1}\right],
\end{equation}
see the proof of \cite{DS1}*{Proposition~3.5}.

Thus, when $J$ is a complete intersection ideal, and $G = \Ok/J$, then using the representation \eqref{eq:ext-ci-isom-explicit}
of $\Ext^p(\Ok/J,\Ok)$, and the representation \eqref{eq:lc-current-repr} of $H^p_Z(\Ok)$, the pairing
\eqref{eq:pairing-grothendieck} is given by
\begin{equation} \label{eq:pairing-ch}
    \la g, h (e_1 \wedge \dots \wedge e_p)^* \ra = g h \dbar\frac{1}{f_p}\wedge \dots \wedge \dbar \frac{1}{f_1}.
\end{equation}

\begin{ex} \label{ex:ci-morphism}
If $I = J(f_1,\dots,f_p)$, and $J = J(g_1,\dots,g_p)$ are two complete intersection ideals of codimension $p$,
then the fact that $I \subseteq J$ is equivalent to the existence of a holomorphic $p\times p$-matrix $A$ such that
$f = A g$. The fact that $I \subseteq J$ means that one has the natural surjection $\pi : \Ok/I \to \Ok/J$.
If one lets $(E,\varphi) = (\bigwedge \Ok^p,\delta_g)$ and $(K,\psi) = (\bigwedge \Ok^p,\delta_f)$ be
the Koszul complexes of $(g_1,\dots,g_p)$ and $(f_1,\dots,f_p)$ respectively, then it is straightforward
to verify that one choice of the morphism $a : (K,\psi) \to (E,\varphi)$ is $a_k := \bigwedge^k A : \bigwedge^k \Ok^p \to \bigwedge^k \Ok^p$.
Hence, using the representation \eqref{eq:ext-ci-isom-explicit} of $\Ext^p(\Ok/I,\Ok)$ and $\Ext^p(\Ok/J,\Ok)$, the
morphism $\Ext^p(\Ok/J,\Ok) \to \Ext^p(\Ok/I,\Ok)$ induced by $\pi : \Ok/I \to \Ok/J$ is given by multiplication
with $a_p = \det A$.
\end{ex}

Using the functoriality of the pairing \eqref{eq:pairing-grothendieck}, and combining this with the expression
\eqref{eq:pairing-ch} for the pairing in the particular case when $g = h = 1$, one gets that
\begin{equation}  \label{eq:transformation-law}
    \dbar \frac{1}{g_p}\wedge \cdots \wedge \dbar \frac{1}{g_1} = (\det A) \dbar \frac{1}{f_p}\wedge \cdots \wedge \dbar \frac{1}{f_1}
\end{equation}
as cohomology classes. Indeed, the \emph{transformation law} for Coleff-Herrera products as proven in \cite{DS1} or \cite{DS2} states that this holds
even as currents.

\section{Reduction to the complete intersection case and linkage} \label{sect:linkage}

It is well-known that for any ideal $J$ of codimension $p$, one can find a complete intersection
ideal $I = J(f_1,\dots,f_p)$ of codimension $p$ contained in $J$, for a proof, see for example
\cite{LarDua}*{Lemma~19}. As a consequence of this well-known fact, we have the following generalization,
which we will make use of in order to reduce properties for the pairing in Theorem~\ref{thm:local-duality}
to the complete intersection case in the previous section.

\begin{lma} \label{lma:complete-intersection-surjective}
    Let $G$ be a finitely generated $\Ok$-module. Then, there exists
    a complete intersection ideal $I$ of codimension $p$, and a
    morphism $\alpha : (\Ok/I)^r \to G$ for some $r \in \N$,
    which is surjective onto $G_{(p)}$.
\end{lma}

We will use this to give a rather elementary proof of the non-degeneracy in
Theorem~\ref{thm:local-duality} when $G$ is of the form $G = \Ok/J$, where $J$ is an
ideal of codimension $\geq p$ by means of the theory of linkage. We will also use the non-degeneracy in
Theorem~\ref{thm:local-duality} to prove Theorem~\ref{thm:colon-modules}.

\begin{proof}
    Since $\ann G_{(p)}$ has codimension $p$, as explained above, we can then find a complete intersection
    ideal $I$ of codimension $p$ contained in $\ann G_{(p)}$. Since $G_{(p)}$ is finitely generated, there
    exists a surjective morphism $\pi : \Ok^r \to G_{(p)}$ for some $r \in \N$. Since $I \subseteq \ann G_{(p)}$,
    $\pi$ induces the surjective morphism $\alpha' : (\Ok/I)^r \to G_{(p)}$, and composing this with
    the inclusion $G_{(p)} \subseteq G$, we obtain the desired morphism $\alpha$.
\end{proof}

The following result about vanishing of $\Ext$, follows from \cite{Eis}*{Proposition~18.4}, and we will use it both
in this section, in the partial proof of Theorem~\ref{thm:local-duality} and it will also be an important part
in the full proof of Theorem~\ref{thm:local-duality-general} in Section~\ref{sect:general-pairing}.

\begin{prop} \label{prop:ext-vanishing}
    Let $G$ and $H$ be finitely generated $\Ok$-modules, and assume that $\ann G + \ann H \neq \Ok$.
    Then $$\depth(\ann G,H) = \min\{ r \mid \Ext^r(G,H) \neq 0 \}.$$
    In particular, if $G$ is a finitely generated $\Ok$-module of codimension $p$,
    then $\Ext^r(G,\Ok) = 0$ for $r < p$.
\end{prop}

This last part is indeed a consequence of the first part, since when $H = \Ok$,
which is Cohen-Macaulay, then $\depth(\ann(G),\Ok) = \codim G$.

\begin{lma} \label{lma:HpZW}
    Let $Z \subseteq W$ be two subvarieties of $(\C^n,0)$ of pure codimension $p$.
    Then the induced map
    \begin{equation} \label{eq:HpZW}
        H^p_Z(\Ok) \to H^p_W(\Ok)
    \end{equation}
    is injective.
\end{lma}

\begin{proof}
    We let $J=J_Z$ and $I=J_W$ be the ideals of holomorphic functions vanishing on $Z$ and $W$ respectively. Thus, $I \subseteq J$,
    and we have a short exact sequence
    \begin{equation*}
        0 \to J^t (\Ok/I^t) \to \Ok/I^t \to \Ok/J^t \to 0.
    \end{equation*}
    Since $J^t (\Ok/I^t)$ has codimension $\geq p$, we get by Proposition~\ref{prop:ext-vanishing}
    and the long exact sequence of $\Ext$ an injection
    \begin{equation*}
        0 \to \Ext^p(\Ok/J^t,\Ok) \to \Ext^p(\Ok/I^t,\Ok).
    \end{equation*}
    Similarly, if $s > t$, we have an injection
    \begin{equation*}
        0 \to \Ext^p(\Ok/J^t,\Ok) \to \Ext^p(\Ok/J^s,\Ok),
    \end{equation*}
    and these two injections together give the injectivity of \eqref{eq:HpZW}.
\end{proof}

\subsection{Partial proofs of Theorem~\ref{thm:local-duality} and Theorem~\ref{thm:colon-modules}} \label{ssect:ci-reduction}

If $I$ is a complete intersection ideal of codimension $p$, and $G = \Ok/I$, then we have an explicit
expression for the pairing \eqref{eq:local-duality-pairing} as given by \eqref{eq:pairing-grothendieck-ci}
or by \eqref{eq:pairing-ch} depending on how one represents $H^p_Z(\Ok)$.

We can now rather easily obtain non-degeneracy in the second argument for general modules of codimension $\geq p$.

\begin{lma} \label{lma:reduction-second}
    Let $G$ be a finitely generated $\Ok$-module of codimension $\geq p$, and let $Z \supseteq \supp G$
    be of pure codimension $p$. Consider a pairing
    \begin{equation*}
        G \times \Ext^p(G,\Ok) \to H^p_Z(\Ok),
    \end{equation*}
    which is functorial in $G$. If the pairing is non-degenerate in the second argument for $G$ of the form
    $G = \Ok/I$, where $I$ is any complete intersection ideal of codimension $p$, then it is non-degenerate
    in the second argument for any finitely generated $\Ok$-module.
\end{lma}

\begin{proof}
    Note first that if the pairing is non-degenerate for $\Ok/I$, then by functoriality, it is also
    non-degenerate for $F = (\Ok/I)^r$. Take now $\alpha : (\Ok/I)^r \to G$ as in Lemma~\ref{lma:complete-intersection-surjective}.
    Assume that $\la g, \xi \ra = 0$ for all $\xi \in \Ext^p(G,\Ok)$. By functoriality and the non-degeneracy for
    $(\Ok/I)^r$, we get that $\alpha^* \xi = 0$.

    We have an exact sequence
    \begin{equation*}
        0 \to H \to (\Ok/I)^r \to G \to 0,
    \end{equation*}
    induced by $\alpha$, where $H$ has codimension $\geq p$. Thus, by the long exact sequence of $\Ext$
    associated to this short exact sequence, and the fact that $\Ext^{p-1}(H,\Ok) = 0$
    by Proposition~\ref{prop:ext-vanishing}, we get an injection
    \begin{equation} \label{eq:piinj}
        \alpha^* :  \Ext^p(G,\Ok) \to \Ext^p( (\Ok/I)^r, \Ok).
    \end{equation}
    Since $\alpha^* \xi = 0$, we thus conclude that $\xi = 0$.
\end{proof}

In order to prove non-degeneracy in the first argument of the pairing, we begin with the following lemma,
generalizing the argument in the proof of Lemma~\ref{lma:ci-nondeg}.

\begin{lma} \label{lma:hom-to-extI}
    Let $H$ be a finitely generated $\Ok$-module of codimension $\geq p$,
    and let $J \subseteq \ann H$ be such that $Z = Z(J)$ has pure codimension $p$.
    Then
    \begin{equation} \label{eq:hom-to-extI}
        \Hom(H,H^p_Z(\Ok)) \cong \Hom(H,\Ext^p(\Ok/J,\Ok)).
    \end{equation}
\end{lma}

\begin{proof}
    It is enough to prove that
    \begin{equation} \label{eq:Jt-isomorphism}
        \Hom(H,\Ext^p(\Ok/J,\Ok)) \stackrel{\cong}{\to} \Hom(H,\Ext^p(\Ok/J^t,\Ok))
    \end{equation}
    for any $t \geq 1$. Consider for $t \geq 1$ the short exact sequence
    \begin{equation*}
        0 \to J (\Ok/J^t) \to \Ok/J^t \to \Ok/J \to 0.
    \end{equation*}
    By the long exact sequence of $\Ext$, and Proposition~\ref{prop:ext-vanishing}, we have an exact sequence
    \begin{equation*}
        0 \to \Ext^p(\Ok/J,\Ok) \to \Ext^p(\Ok/J^t,\Ok) \to \Ext^p(J(\Ok/J^t),\Ok).
    \end{equation*}
    Hence, by left exactness of $\Hom(H,\bullet)$, \eqref{eq:Jt-isomorphism} is injective, and it
    remains to prove that it is surjective.
    Consider thus $\beta \in \Hom(H,\Ext^p(\Ok/J^t,\Ok))$. If $h \in H$, then
    $J \beta(h) = 0$. If one represents $\Ext$ with the help of an injective resolution
    in the second argument, one sees that the image of $\beta(h)$ in $\Ext^p(J(\Ok/J^t),\Ok)$ is $0$,
    so $\beta(h)$ lifts to a unique element in $\Ext^p(\Ok/J,\Ok)$, and $\beta$ thus lifts to
    a morphism $\Hom(H,\Ext^p(\Ok/J,\Ok))$.
\end{proof}

If we take $H = \Ext^p(G,\Ok)$ in Lemma~\ref{lma:hom-to-extI}, it is thus enough to prove non-degeneracy in the first argument
for the Yoneda pairing, without composing with $\pi_t$ from \eqref{eq:pi-t}.

We will also use the following alternative description of the module $I : J$ which
appears in Theorem~\ref{thm:colon-modules}. To begin with, we set the notation
which we will use throughout the rest of this section.
We assume that $G$ is a finitely generated $\Ok$-module of codimension $\geq p$.
We let
\begin{equation*}
    \alpha : (\Ok/I)^r \to G
\end{equation*}
be a surjective morphism as in Lemma~\ref{lma:complete-intersection-surjective},
where $I$ is a complete intersection ideal of codimension $p$ contained in $\ann G$.
Let $(L,\delta_f)$ be the Koszul complex of a minimal set $(f_1,\dots,f_p)$ of generators
of $I$, and we let $e_1,\dots,e_p$ be the standard basis of $\Ok^r$, so that $e := e_1 \wedge \dots \wedge e_p$
is a basis element of $L_p$, the $p$-th term in the complex.
We let $(K,\psi) = (L \otimes \Ok^r, \delta_f \otimes \Id_{\Ok^r})$ be the direct sum of $r$ copies of $(L,\delta_f)$,
which is a free resolution of $(\Ok/I)^r$, and finally, we let
\begin{equation*}
    a : (K,\psi) \to (E,\varphi)
\end{equation*}
be a morphism of complexes extending $\alpha$, as in Proposition~\ref{prop:comparison-morphism}.
To summarize, we have a commutative diagram as follows:
\begin{equation*}
\xymatrix{
    E_{p+1} \ar[r]^{\varphi_{p+1}} & E_p \ar[rr]^{\varphi_p} &  & \cdots  \ar[rr]^{\varphi_1} & & E_0 \ar[r] & G \cong \Ok^r/J \\
    0 \ar[r] & \ar[u]^{a_p} K_p = L_p \otimes \Ok^r \ar[rr]^{\delta_f \otimes \Id_{\Ok^r}} & \hspace{0.2cm} & \cdots \ar[rr]^{\delta_f \otimes \Id_{\Ok^r}} & \hspace{1pt} & L_0 \otimes \Ok^r \ar[u]^{a_0}  \ar[r] & (\Ok/I)^r \ar[u]^{\alpha}.
% extra columns with \hspace to make arrows fit
}
\end{equation*}

\begin{lma} \label{lma:colon}
    Let $G \cong \Ok^r/J$, where $G$ has codimension $\geq p$, and let $(E,\varphi)$, $(K,\psi)$, $(L,\delta_f)$
    and $a : (K,\psi) \to (E,\varphi)$ be as above.
    Then,
    \begin{equation*}
        I : (L_p \otimes J) = I K_p^* + (\ker \varphi_{p+1}^*) a_p.
    \end{equation*}
\end{lma}

Note that if $G = \Ok/J$ is Cohen-Macaulay of codimension $p$, and $(E,\varphi)$ has length $p$, then
$(\ker \varphi_{p+1}^*) a_p$ is simply the ideal $J(a_p)$ generated by the entries of $a_p$, i.e.,
\begin{equation*}
    I : J = I + J(a_p).
\end{equation*}
In this case, Lemma~\ref{lma:colon} is well-known, and can be found for example in Lemma~3.2 in \cite{FH}.
The following is an adaption of this proof to our more general situation.

\begin{proof}
    We let $M := (\ker \varphi_{p+1}^*) a_p$.
    First of all, we prove that $I K_p^* + M \subseteq I: (L_p \otimes J)$. It is clear that $I K_p^* \subseteq I : (L_p \otimes J)$,
    so we want to prove that $M \subseteq I:(L_p \otimes J)$. Take $g \in J \subseteq \Ok^r$.
    The element $g$ induces a morphism $\Hom( \Ok/I, (\Ok/I)^r)$, which extends to
    the morphism of complexes
    \begin{equation*}
        \Id_L \otimes g : L \to L \otimes \Ok^r \cong K.
    \end{equation*}

    Since $g \in J$, we get that $a ( \Id_L \otimes g ) : (L,\delta_f) \to (E,\varphi)$
    is a morphism of complexes extending the zero morphism $\Ok/I \to G$.
    Thus, by Proposition~\ref{prop:comparison-morphism}, there exists a homotopy
    $s : (L,\delta_f) \to (E,\varphi)$ between $0$ and $a (\Id_L \otimes g)$. In particular,
    \begin{equation} \label{eq:g-homotopy}
        a_p (\Id_{L_p} \otimes g) = s_{p-1} (\delta_f)_p + \varphi_{p+1} s_p.
    \end{equation}
    If $\xi \in \ker \varphi_{p+1}^*$, and we apply this to \eqref{eq:g-homotopy}, we get that
    \begin{equation*}
        \xi a_p ( \Id_{L_p} \otimes g ) = \xi s_{p-1} (\delta_f)_p \subseteq I,
    \end{equation*}
    since $\im (\delta_f)_p \subseteq I L_{p-1}$. Hence, $M \subseteq I : (L_p \otimes J)$.

    Conversely, we consider an element $\gamma \in I : (L_p \otimes J)$. By the isomorphism $L_p \cong \Ok$
    (given by $e^*$, the dual of the basis $e = e_1 \wedge \dots \wedge e_p$ of $L_p$),
    we can consider $\gamma$ as an element $\tilde{\gamma} : \Ok^r \to \Ok$, and we have that
    \begin{equation} \label{eq:gammatilde}
        \gamma = e^*(\Id_{L_p} \otimes \tilde{\gamma} ).
    \end{equation}
    The morphism $\tilde{\gamma}$ descends to a morphism $\Ok^r/J \to \Ok/I$, and by Proposition~\ref{prop:comparison-morphism},
    we can find a morphism of complexes $b : (E,\varphi) \to (L,\delta_f)$ extending this morphism.
    The morphism $\tilde{\gamma}$ also induces a morphism $(\Ok/I)^r \to \Ok/I$, which in turn induces a morphism
    of complexes $(K,\psi) \to (L,\delta_f)$, which is given simply as
    \begin{equation*}
         \Id_L \otimes \tilde{\gamma} : K \cong L \otimes \Ok^r \to L \otimes \Ok \cong L.
    \end{equation*}
    Then, $\Id_L \otimes \tilde{\gamma}$ and $ba$ both extend the morphism $(\Ok/I)^r \to \Ok/I$ induced by
    $\tilde{\gamma}$, so by Proposition~\ref{prop:comparison-morphism}, $b a$ is homotopic to
    $\Id_L \otimes \tilde{\gamma}$, and in particular, there exists $s_{p-1}$ such that
    \begin{equation} \label{eq:ghomot}
        \Id_{L_p} \otimes \tilde{\gamma}- b_p a_p = s_{p-1} \psi_p.
    \end{equation}
    In addition, since $b : (E,\varphi) \to (L,\delta_f)$ is a morphism of complexes, and $L_{p+1} = 0$, we get that
    $\varphi_{p+1}^* b_p = 0$. Thus, $e^* b_p \in \ker \varphi_{p+1}^*$.
    To conclude, applying $e^*$ to \eqref{eq:ghomot}, and using this in combination with \eqref{eq:gammatilde}, we get that
    \begin{equation*}
        \gamma = (e^* b_p) a_p + e^* s_{p-1} \psi_p \subseteq M + I K_p^*,
    \end{equation*}
    and we have proven the other inclusion.
\end{proof}

In order to prove non-degeneracy in the first argument, one cannot as easily reduce non-degeneracy to
the complete intersection case, but using \eqref{eq:link-ideals} from the theory of linkage,
we can do this reduction when $G$ is of the form $G = \Ok/J$, where $G$ has codimension $\geq p$.

In the case of non-pure dimension, we first relate $G^{(p)}$ in Theorem~\ref{thm:local-duality}
and $J_{[p]}$ in \eqref{eq:link-ideals} or \eqref{eq:link-modules}.

\begin{remark} \label{rem:equihull}
    If $G = \Ok^r/J$, where $G$ has codimension $\geq p$, then we claim that
    $(\Ok^r/J)^{(p)} = \Ok^r/J_{[p]}$.
    To see this, note that if $g \in (\Ok^r/J)_{(p+1)}$, then $g \in J_{[p]}$.
    Thus, we get a well-defined surjective map $(\Ok^r/J)^{(p)} \to \Ok^r/J_{[p]}$.
    In addition, it is injective, since if $g = 0$ in $\Ok^r/J_{[p]}$, and if we write
    $J = J_{[p]} \cap J_{[\geq p+1]}$, then $g \in J$ outside of $\supp \Ok^r/J_{[\geq p+1]}$
    which has codimension $\geq p+1$, and thus, $g \in (\Ok^r/J)_{(p+1)}$.
\end{remark}

\begin{lma} \label{lma:reduction-first}
    Let $G = \Ok/J$, where $J$ has codimension $\geq p$, and let $Z \supseteq \supp G$
    be of pure codimension $p$. Consider a pairing
    \begin{equation*}
        G \times \Ext^p(G,\Ok) \to H^p_Z(\Ok),
    \end{equation*}
    which is functorial in $G$. If the descended pairing
    \begin{equation*}
        G/G_{(p+1)} \times \Ext^p(G,\Ok) \to H^p_Z(\Ok),
    \end{equation*}
    is non-degenerate in the second argument for $G$ of the form $G = \Ok/I$, where $I$ is any complete
    intersection ideal of codimension $p$, then it is non-degenerate in the first argument for any
    $G = \Ok/J$, where $J$ has codimension $\geq p$.
\end{lma}

\begin{proof}
    We let $(E,\varphi)$ be a free resolution of $G = \Ok/J$, and using the representation
    $\Ext^p(G,\Ok) \cong H^p(\Hom(E_\bullet,\Ok))$, we write any $\xi \in \Ext^p(G,\Ok)$
    as $\xi = [\xi_0]$, where $\xi_0 \in \ker \varphi_{p+1}^*$.
    We let $I \subseteq J$ be a complete intersection ideal contained in $J$, $(K,\psi)$
    the Koszul complex of a set of minimal generators of $I$, and let $a : (K,\psi) \to (E,\varphi)$
    be a morphism of complexes extending the natural surjection $\pi : \Ok/I \to \Ok/J$.

    Take $g \in G$ such that $\la g, \xi \ra = 0$ for all $\xi$.
    We want to show that $g \in G_{(p+1)}$.
    For $f \in \Ok/I$, we also get that $\la fg, \xi \ra = 0$ for all $\xi$.
    Thus, by $\Ok$-linearity, functoriality, and using the representation of $\Ext$ above, we get that
    \begin{equation*}
        \la f, [a_p^* \xi_0 g] \ra = 0
    \end{equation*}
    for all $f \in \Ok/I$, and all $\xi_0 \in \ker \varphi_{p+1}$. By non-degeneracy in the second argument,
    for $\Ok/I$, we get that $\xi_0 a_p g = 0$ in $\Ext^p(\Ok/I,\Ok)$. Thus,
    $g \in I : I+(\ker \varphi_{p+1}^*) a_p$, and by Lemma~\ref{lma:colon},
    $g \in I : (I:J)$, so, $g \in J_{[p]}$ by \eqref{eq:link-ideals}, i.e., $g \in G_{(p+1)}$.
\end{proof}

Arguing in a similar way, one would also obtain that \eqref{eq:link-modules} implies
non-degeneracy in the first argument for any finitely generated $\Ok$-module $G$
of codimension $\geq p$, and not just $G$ of the form $G = \Ok/J$. However, since we
know of a proof of \eqref{eq:link-modules} which does not depend on Theorem~\ref{thm:local-duality}
only in the case when $G = \Ok/J$, i.e., \eqref{eq:link-ideals}, we have only stated
Lemma~\ref{lma:reduction-first} in this case.

In the proof of Lemma~\ref{lma:reduction-first}, we used Lemma~\ref{lma:colon}
and \eqref{eq:link-ideals} to obtain a proof of non-degeneracy in the first argument of the pairing
in Theorem~\ref{thm:local-duality} when $G$ is of the form $G = \Ok/J$.
Here we show that we can also go the other way, and prove \eqref{eq:link-ideals}, or more generally
\eqref{eq:link-modules} using the functoriality and non-degeneracy in the first argument of the
pairing in Theorem~\ref{thm:local-duality}.

\begin{proof}[Proof of Theorem~\ref{thm:colon-modules}]
    It is clear that $J_{[p]} \subseteq I:(I:J_{[p]})$. In addition, we claim that $I:J \subseteq I:J_{[p]}$,
    which implies that $J_{[p]} \subseteq I:(I:J)$. To prove the claim, we first write
    $J = J_{[p]} \cap J_{[\geq p+1]}$, where $J_{[\geq p+1]}$ is the intersection of all primary components
    of codimension $\geq p+1$. If $\gamma(J) \subseteq I$, then $\gamma(J_{[p]}) \subseteq I$ outside of $Z(J_{[\geq p+1]})$
    which has codimension $\geq p+1$, so $\gamma(J_{[p]}) \subseteq (\Ok/I)_{(p+1)}$, and thus, $\gamma(J_{[p]}) = 0$ in $\Ok/I$,
    since $\Ok/I$ has pure codimension $p$.

    It remains to prove the reverse inclusion $I:(I:J) \subseteq J_{[p]}$.
    We now take $(E,\varphi)$, $(K,\psi)$, $(L,\delta_f)$ and $a : (K,\psi) \to (E,\varphi)$
    as before the statement of Lemma~\ref{lma:colon}.
    Take $g \in I : (I:J)$, considered as an element in $\Ok^r$. This element induces a morphism $\epsilon_g : \Hom(\Ok/I,(\Ok/I)^r)$,
    and we let $c : (L,\delta_f) \to (K,\psi)$ be the morphism
    \begin{equation*}
        \Id_L \otimes g : L \to \Ok^r \otimes L = K
    \end{equation*}
    which extends $\epsilon_g$. By the fact that $g \in I : (I:J)$, we get that
    $c_p^* \in (I : (I : L_p \otimes J))$.
    We note that $g = \alpha \epsilon_g(1)$, so by functoriality of the pairing, we get
    \begin{equation*}
        \la g, \xi \ra = \la \alpha(\epsilon_g(1)), \xi \ra = \la 1, \epsilon_g^* \alpha^* \xi \ra.
    \end{equation*}
    If we represent $\xi = [\xi_0]$, where $\xi_0 \in \ker \varphi_{p+1}^*$, then
    \begin{equation*}
        [\epsilon_g^* \alpha^* \xi_0 ] = [ \xi_0 a_p c_p ].
    \end{equation*}
    Since $c_p^* \in (I : (I : L_p \otimes J)) = I : (I K_p^* + (\ker \varphi_{p+1}^*) a_p)$ by Lemma~\ref{lma:colon},
    we get that $\im \xi_0 a_p c_p \in I$, so $[ \xi_0 a_p c_p ] = 0$, since
    $I \Ext^p(\Ok/I,\Ok) = 0$. Thus, $\la g, \xi \ra = 0$ for all $\xi$, and by non-degeneracy
    of the pairing \eqref{eq:local-duality-pairing-descended}, we get that $g = 0$ in $G^{(p)}$,
    i.e., $g \in J_{[p]}$ by Remark~\ref{rem:equihull}.
\end{proof}

\section{Proofs of Theorem~\ref{thm:local-duality}, Theorem~\ref{thm:local-duality-surjective} and Theorem~\ref{thm:local-duality-general}} \label{sect:general-pairing}

We begin with the following consequence of Proposition~\ref{prop:ext-vanishing}.

\begin{cor} \label{cor:extp-left-exact}
    Let
    \begin{equation*}
        A \to B \to C \to 0
    \end{equation*}
    be an exact sequence of finitely generated $\Ok$-modules such that
    $A$, $B$ and $C$ have support in a variety of codimension $\geq p$.
    Then
    \begin{equation*}
        0 \to \Ext^p(C,\Ok) \to \Ext^p(B,\Ok) \to \Ext^p(A,\Ok)
    \end{equation*}
    is exact.
\end{cor}

\begin{proof}
    We let $K = \ker (B \to C)$, and thus have a short exact sequence
    \begin{equation*}
        0 \to K \to B \to C \to 0,
    \end{equation*}
    and since $\codim K \geq p$, we get from the long exact sequence of $\Ext$
    and Proposition~\ref{prop:ext-vanishing} an exact sequence
    \begin{equation} \label{eq:ExtCBK}
        0 \to \Ext^p(C,\Ok) \to \Ext^p(B,\Ok) \to \Ext^p(K,\Ok).
    \end{equation}
    Since $K = \ker (B \to C) = \im (A \to B)$, if we let $A' := \ker(A \to B)$,
    we have a short exact sequence
    \begin{equation*}
       0 \to A' \to A \to K \to 0,
    \end{equation*}
    and since $A'$ has codimension $\geq p$, we get as above an injection
    \begin{equation*}
        \Ext^p(K,\Ok) \to \Ext^p(A,\Ok).
    \end{equation*}
    Composing \eqref{eq:ExtCBK} with this injection at the end, we get the desired short exact sequence.
\end{proof}

We will combine Corollary~\ref{cor:extp-left-exact} with the following
result, which is a combination of \cite{HLocal}*{Lemma~4.1} and \cite{HLocal}*{Proposition~4.2}.

\begin{prop} \label{prop:grothendieck-isomorphism}
    Let $R$ be a commutative Noetherian ring, and let $T$ be a contravariant additive functor from the
    category of finitely generated $R$-modules to the category of abelian groups.
    Then there is a natural transformation of functors
    \begin{equation} \label{eq:functorial-isomorphism}
        T \to \Hom(\bullet,T(R)),
    \end{equation}
    and \eqref{eq:functorial-isomorphism} is an isomorphism if and only if $T$ is left exact.
\end{prop}

If we let $R = \Ok/J$, where $J$ has codimension $p$, then combining these two results, we get the following corollary,
since by Corollary~\ref{cor:extp-left-exact}, the functor $T(G) := \Ext^p(G,\Ok)$
is left exact on finitely generated $\Ok/J$-modules.

\begin{cor} \label{cor:exthomp}
    Let $G$ be a finitely generated $\Ok$-module, $J \subseteq \ann G$,
    and assume that $J$ is of codimension $\geq p$. Then
    \begin{equation*}
        \Ext^p(G,\Ok) \to \Hom(G,\Ext^p(\Ok/J,\Ok))
    \end{equation*}
    is an isomorphism.
\end{cor}

From the proof of Proposition~\ref{prop:grothendieck-isomorphism}, it follows
that the morphism sends $\xi$ to $g \mapsto \epsilon_g^* \xi$.

\subsection{Singularity subvarieties and the $S_k$-property} \label{ssect:Zp-Sk}

The statements of Theorem~\ref{thm:local-duality-surjective} and Theorem~\ref{thm:local-duality-general} involve
Serre's $S_2$-property of a finitely generated $\Ok$-module
$G$, which can be expressed in terms of the following singularity subvarieties
associated to $G$. Although the $S_2$-property is defined for arbitrary
finitely generated $\Ok$-modules, we will here only treat the case of
$\Ok$-modules $G$ such that $Z(G)$ has pure codimension $p$,
as the definition becomes easier in this case, and this case is the
only one that we will use.

Given a free resolution $(E,\varphi)$ of $G$,
the associated singularity subvariety $Z_\ell = Z^E_\ell$ is
defined as the subvariety where $\varphi_\ell$ does not have optimal rank, where $\ell = 1,2,\dots$.

If $G$ is such that $Z = Z(G)$ has pure codimension $p$, then $G$ is said to be $S_k$ if
\begin{equation} \label{eq:Sk-def}
    \codim Z_{\ell} \geq \ell + k \text{ for } \ell \geq p+1.
\end{equation}

\begin{remark} \label{rem:SkST}
In \cite{ST}, certain singularity subvarieties $S_m$ associated to a coherent analytic sheaf
are defined, which are related to our singularity subvarieties $Z_k$ by the simple
relation that if $\dim X = n$, then $Z_k = S_{n-k}$, see \cite{LarDua}*{Proposition~26}.
\end{remark}

In commutative algebra, the $S_k$-condition is usually phrased as follows: $G$ is $S_k$ if
\begin{equation} \label{eq:Sk-alt-def}
    \depth G_P \geq \min(k,\dim G_P),
\end{equation}
for any prime ideal $P$, where $G_P$ denotes the localization at $P$. This is what appears to be the standard 
definition of the $S_k$-property for modules or sheaves, and appears in this way in for example \cite{BH}*{p. 62}.
In for example \cite{HGenDiv}*{p. 291}, another inequivalent definition is used, where $\dim G_P$ in the right-hand
side of \eqref{eq:Sk-alt-def} is replaced by $\dim \Ok_P$.

\begin{lma} \label{lma:altSkdefs}
    If $G$ has pure dimension $p$, then the conditions \eqref{eq:Sk-def} and \eqref{eq:Sk-alt-def} are equivalent.
\end{lma}

By for example the proof of \cite{BS}*{Theorem~II.2.1},
\begin{equation} \label{eq:Ext-Zk}
    Z_\ell(G) = \cup_{r \geq \ell} \supp \Ext^r(G,\Ok).
\end{equation}
Then, Lemma~\ref{lma:altSkdefs} follows from \cite{HL}*{Proposition~1.1.6} and \eqref{eq:Ext-Zk}.

Hence, using \eqref{eq:Ext-Zk}, if $Z(G)$ has pure codimension $p$, then $G$ is $S_k$ if 
and only if
\begin{equation} \label{eq:Sk-Ext-def}
    \codim \supp \Ext^\ell(G,\Ok)  \geq \ell + k \text{ for } \ell \geq p+1.
\end{equation}
In particular, when $Z(G)$ has pure codimension $p$, then by Proposition~\ref{prop:ext-ass-primes}, $G$ has pure codimension $p$ if
and only if $G$ is $S_1$. (The $S_1$-property is defined also for $G$ such that $Z(G)$ does not have pure
dimension, and then $S_1$ should mean that $G$ has no embedded primes.)

If $G = \Ok_X/J$, where $J = J_Z$, the ideal
of holomorphic functions vanishing on a subvariety $Z \subseteq X$,
then by the Serre criterion for normality, $Z$ is normal if and only
if $G$ is $S_2$ and $R_1$ (where $R_1$ means that $Z_{\sing}$ has
codimension at least $2$ in $Z$), cf., for example \cite{Markoe}*{Theorem~1}.

The following consequence of Proposition~\ref{prop:ext-kp} and \eqref{eq:Sk-Ext-def}
will be important in the proof of the main theorems.

\begin{cor} \label{cor:extisS2}
    If $G$ is a finitely generated $\Ok$-module of codimension $p$, then
    $\Ext^p(G,\Ok)$ is $S_2$.
\end{cor}

We will later on use the following characterization of the $S_2$-property, \cite{ST}*{Theorem~1.14} and Remark~\ref{rem:SkST}.

\begin{thm}
    Let $\mathcal{G}$ be a coherent analytic sheaf of pure codimension $p$.
    Then $\mathcal{G}$ is $S_2$ if and only if any local section of $\mathcal{G}$ defined outside some subvariety $Z$
    of codimension $\geq p+2$ extends over $Z$.
\end{thm}

From this, we obtain the following corollary.

\begin{cor} \label{cor:S2-isom}
    Let $F,G$ and $H$ be finitely generated $\Ok$-modules, and assume that there is an exact sequence
    \begin{equation*}
        0 \to F \to G \to H \to 0,
    \end{equation*}
    where $F$ and $G$ have pure codimension $p$, and $H$ has codimension $\geq p+2$. If $F$ is $S_2$, then $F \cong G$.
\end{cor}

\subsection{Dualizing complexes}

We will make use of the concept of dualizing complex of Grothendieck, here
formulated in the more basic language, avoiding derived categories, as done by Sharp,
\cite{Sharp}. See also for example \cite{Iver}*{Chapter~7} for a treatment of the necessary material.

Let $R$ be a Noetherian commutative ring, and let $C$ and $D$ be two complexes of $R$-modules, with differentials $d_C$ and $d_D$.
The complex $\Hom(C,D)$ is defined as $$\Hom^{k}(C,D) = \oplus_{\ell} \Hom(C_\ell,D_{\ell+k}),$$ where
the differential is $d \phi = d_D \phi - (-1)^k \phi d_C$ (or some other suitable choice of sign convention).
Recall that a morphism of complexes $a : C \to D$ is a quasi-isomorphism if the induced morphism $a_* : H^k(C) \to H^k(D)$
is an isomorphism for all $k$.

If $C$ and $D$ are complexes, one can define a natural morphism of complexes
\begin{equation} \label{eval}
    C \to \Hom(\Hom(C,D),D),
\end{equation}
which if $C$ is the complex, $C_0 = G$ and $C_k = 0$, $k \neq 0$,
then \eqref{eval} is simply the evaluation map $g \mapsto (\psi \mapsto \psi(g))$.
A \emph{dualizing complex} for $R$ is a bounded complex $D$ of injective modules with finitely
generated cohomology,  such that for all bounded complexes $C$ with finitely generated cohomology,
the map \eqref{eval} is a quasi-isomorphism.

If $R$ is a regular local ring, then $R$ has a finite injective resolution $I$, and
$I$ is then a dualizing complex for $R$, see for example \cite{Iver}*{Corollary~7.20}. 
In particular, for $R = \Ok$, we could take $I$ to be the Dolbeault complex of
germs of $(0,*)$-currents.

It also follows easily that if $I$ is a dualizing complex for $R$, then $\Hom(R/J,I)$ is a
dualizing complex for $R/J$, see \cite{Iver}*{Proposition~7.25}, which
implies the following result.

\begin{prop} \label{prop:dualizing}
    Let $J \subseteq \Ok$ be an ideal, $I$ a finite injective resolution of $\Ok$,
    and $I_J := \Hom(\Ok/J,I)$. Then, for any finitely generated $\Ok/J$-module $G$,
    the evaluation map \eqref{eval} gives an isomorphism
    \begin{equation} \label{eq:dualizing}
        G \to H^0(\Hom(\Hom(G,I_J),I_J)).
    \end{equation}
\end{prop}

\subsection{Proofs of non-degeneracy and surjectivity of induced morphisms in Theorem~\ref{thm:local-duality}
and \ref{thm:local-duality-surjective}} \label{ssect:proof-general1}

The non-degeneracy in the first argument, as well as the isomorphism \eqref{eq:general-duality-surjection1}
when $G$ is $S_2$ is a rather simple consequence of the following result, in combination with Proposition~\ref{prop:dualizing}.
Throughout this section, for an ideal $J$, we let $I_J$ be as in Proposition~\ref{prop:dualizing}.

\begin{prop} \label{prop:ssreplacement}
    Let $G$ be an $\Ok$-module and $J \subseteq \ann G$, and assume that $G$ and $J$ have pure codimension $p$.
    Let $\Psi$ be the map
    \begin{align} \label{eq:Psi}
        H^0( \Hom(\Hom(G,I_J),I_J))  
        \to \Hom(H^p(\Hom(G,I_J)),H^p(I_J)),
    \end{align}
    which if $[\psi] \in H^0( \Hom(\Hom(G,I_J),I_J))$ and $[\phi] \in H^p(\Hom(G,I_J))$,
    and if we write $\psi = \psi_0 + \dots + \psi_n$, where $\psi_k \in \Hom(\Hom(G,I_J^k),I_J^k)$,
    then $\Psi([\psi]) : [\phi] \mapsto [\psi_p \phi]$.
    Then $\Psi$ is injective, and if 
    $G$ is $S_2$, then $\Psi$ is surjective.
\end{prop}

The fact that there is such an injective map, \eqref{eq:Psi}, which is surjective if $G$ is $S_2$
follows from the convergence of one of the spectral sequences of the double complex
$\Hom(\Hom(G,I_J),I_J)$. However, it is not clear to us how to extract from such
an argument that the morphism is the one as claimed in the proposition,
and hence we give a more direct proof of Proposition~\ref{prop:ssreplacement},
which corresponds to proving directly the convergence of the spectral sequence in
this special situation.

\begin{proof}
    We consider the double complex $E_{s,t}$ as defined by $$E_{s,t} = \Hom(\Hom(G,I_J^s),I_J^t),$$
    and we let $\partial'$ and $\partial''$ be the differentials when $t$ and $s$ are fixed, respectively,
    i.e.,
    $$
        \partial' : E_{s+1,t} \to E_{s,t} \text{ and } \partial'' : E_{s,t} \to E_{s,t+1}.
    $$
    We denote by $A$ the total complex of $E$, $A_k = \oplus E_{s,s+k}$, and
    where the differential on $E_{s,t}$ is $\partial = \partial' + (-1)^{s-t} \partial''$.

    We begin by verifying that $\Psi$ is indeed well-defined. Let $\psi \in H^0( A )$.
    Thus, if we write $\psi = \psi_0 + \dots + \psi_n$, then $\psi \in Z_0(A)$ is equivalent
    to that $\partial'' \psi_k = -\partial' \psi_{k+1}$ for $k=0,\dots,n-1$.
    Using this, it is straight-forward to verify that $\Psi$ is well-defined, i.e., if $\phi \in Z_p(\Hom(G,I_J))$,
    then $\dbar \psi_p \phi = 0$, and if $\phi \in B_p(\Hom(G,I_J))$, $\phi = \dbar \eta$, then 
    $\psi_p \phi = \dbar( -\psi_{p-1}\eta)$.

    We let $F_{s,t} = H^s_{\partial'}(E_{\bullet,t})$
    (which thus corresponds to the $E_1$-page for the spectral sequence associated to the double complex $E_{s,t}$).
    Since $G$ is annihilated by $J$, 
    $$\Hom(\Hom(G,I_J^s),I_J^t) \cong \Hom(\Hom(G,I^s),I^t),$$
    and since $I^t$ is injective, we get that $$F_{s,t} \cong \Hom(\Ext^s(G,\Ok),I^t).$$

    We let $H_{s,t} = H^t_{\partial''} (F_{s,\bullet})$ (which thus corresponds to the $E_2$-page for the spectral sequence
    associated to the double complex $E_{s,t}$), and we note that
    \begin{equation} \label{eq:Hext}
        H_{s,t} = \Ext^t(\Ext^s(G,\Ok),\Ok).
    \end{equation}

    \begin{lma} \label{lma:maptoHkk}
        Let $k \geq 0$, let $[\psi] \in H^0(A)$, and assume that
        $\psi = \psi_k + \dots + \psi_n$, where $\psi_\ell \in \Hom(\Hom(G,I_J^\ell),I_J^\ell)$.
        Then, there is a well-defined map $[\psi] \mapsto [ [\psi_k ]] \in H_{k,k}$. If $[ [\psi_k ] ] = 0$, then
        \begin{equation} \label{eq:psireduction}
        [\psi] = [\tilde{\psi}_{k+1} + \psi_{k+2} + \dots + \psi_n],
        \end{equation}
        where $\tilde{\psi}_{k+1} = \psi_{k+1} + \partial''\eta$ for some $\eta \in E_{k+1,k}$.
    \end{lma}

    \begin{proof}
        First of all, we note that if $\psi = \psi_k + \dots + \psi_n$, then $\partial' \psi_k = 0$,
        so $\psi_k$ defines a class $[\psi_k] \in F_{k,k}$. Since $\partial''\psi_k = -\partial' \psi_{k+1}$,
        $[\psi_k]$ defines a class $[ [\psi_k ]] \in H_{k,k}$.

        We then verify that the map is well-defined, so we assume that $\psi = \psi_k + \dots + \psi_n$ and $[\psi] = 0$.
        Then, $\psi_k = \partial'' \gamma + \partial' \eta$ for some $\gamma \in E_{k,k-1}$
        with $\partial' \gamma = 0$, and some $\eta \in E_{k+1,k}$, and this implies that $[ [\psi_k] ] = 0$.

        Finally, if $[ [\psi_k] ] = 0$, then $\psi_k = \partial'' \gamma + \partial' \eta$, where $\gamma$ and
        $\eta$ are as above, and we conclude that with this choice of $\eta$, \eqref{eq:psireduction} holds.
    \end{proof}

    By \eqref{eq:Hext} and Proposition~\ref{prop:ext-ass-primes}, since $G$ has pure codimension $p$,
    \begin{equation} \label{eq:Hkk-vanish}
        H_{k,k} = 0 \text{ for $k \neq p$. }
    \end{equation}

    We now define a map
    \begin{equation} \label{eq:tildePsi}
        \tilde{\Psi} : H^0(A) \to H_{p,p}
    \end{equation}
    as follows. By \eqref{eq:Hkk-vanish} and Lemma~\ref{lma:maptoHkk} used repeatedly for
    $k=0,\dots,p-1$, an element $[\psi] \in H^0(A)$, where $\psi = \psi_0 + \dots + \psi_n$
    equals $[\tilde{\psi}_p + \psi_{p+1} + \dots + \psi_n]$, where 
    $\tilde{\psi}_p = \psi_p + \dbar \eta$ for some $\eta \in E_{p,p-1}$.
    By Lemma~\ref{lma:maptoHkk}, the map $\tilde{\Psi}([\psi]) := [[\tilde{\psi}_p]]$
    is well-defined.

    \begin{lma} \label{lma:tildePsiproperties}
        If $G$ and $J$ have pure codimension $p$, then $\tilde{\Psi}$ in \eqref{eq:tildePsi} is injective,
        and if $G$ is $S_2$, then $\tilde{\Psi}$ is surjective.
    \end{lma}

    \begin{proof}
        We begin by proving injectivity. If $[\psi] = [\psi_p + \dots + \psi_n]$,
        and $[ [\psi_p] ] = 0$, then by Lemma~\ref{lma:maptoHkk}, $[\psi] = [\tilde{\psi}_{p+1} + \psi_{p+2} + \dots + \psi_n]$.
        Since $H_{k,k} = 0$ for $k > p$ by \eqref{eq:Hkk-vanish}, we can continue this argument by induction to obtain that $[\psi] = 0$.

        To prove surjectivity, we take $[[\psi_p]]$ in $H_{p,p}$, and we want to prove that there exist
        $\psi_k \in E_{k,k}$ for $k=p+1,\dots,n$ such that $\partial(\psi_p + \dots + \psi_n) = 0$.
        First of all, since $[[\psi_p]]$ defines a class in $H_{p,p}$, $\partial' \psi_p = 0$
        and $\partial'' \psi_p = - \partial' \eta_{p+1}$ for some $\eta_{p+1}$ in $E_{p+1,p+1}$.
        
        By commutativity, $\partial'' \eta_{p+1}$ defines a class $[\partial'' \eta_{p+1}]$ in $F_{p+1,p+2}$,
        which is $\partial''$-closed, so it defines a class $[[\partial'' \eta_{p+1}]]$ in $H_{p+1,p+2}$.
        By Proposition~\ref{prop:ext-vanishing}, \eqref{eq:Sk-Ext-def} and \eqref{eq:Hext}, $H_{p+1,p+2} = 0$, since $G$ is $S_2$.
        Hence, $[\partial'' \eta_{p+1}] = \partial''[\gamma_{p+1}]$, where $\partial' \gamma_{p+1} = 0$,
        and we can take $\psi_{p+1}: = \eta_{p+1} - \gamma_{p+1}$. Then, $[\partial'' \psi_{p+1}] = 0$,
        and continuing in a similar way, we can by induction find the desired $\psi_{p+2},\dots,\psi_n$,
        using that $H_{p+k,p+k+1} = 0$ for $k=2,\dots$.
    \end{proof}

    From the definition of $\tilde{\Psi}$, and the fact that $(\partial'' \eta)_* = 0$ in $\Hom(H^p(G,I_J),H^p(I_J))$,
    it follows that $\Psi$ is the composition of $\tilde{\Psi}$
    with the following isomorphism:
    \begin{align*}
        H^p_{\partial''}(H^p_{\partial'}(\Hom(\Hom(G,I_J),I_J))) &\stackrel{\cong}{\to} H^p(H^p(\Hom(G,I_J)),I_J)   \\
        &\stackrel{\cong}{\to} \Hom(H^p(\Hom(G,I_J)),H^p(I_J)),
    \end{align*}
    where the first map is an isomorphism since $I_J$ is injective, and the second is an isomorphism
    by Corollary~\ref{prop:grothendieck-isomorphism}, since $$H^p(\Hom(G,I_J)) \cong H^p(\Hom(G,I)) \cong \Ext^p(G,\Ok),$$
    which has codimension $\geq p$ by Proposition~\ref{prop:ext-ass-primes}.
    The proposition then follows by this description of $\Psi$, and Lemma~\ref{lma:tildePsiproperties}.
\end{proof}

    \begin{lma} \label{lma:reduction-to-ext-pure}
        Let $G$, $Z$ and $p$ be as in Theorem~\ref{thm:local-duality},
        and let $J$ be an ideal of pure codimension $p$ such that $Z(J) = Z$.

        \noindent 1) The pairing \eqref{eq:local-duality-pairing-descended} is non-degenerate if and only if the pairing
        \begin{equation} \label{eq:pairing-to-extI}
            G/G_{(p+1)} \times \Ext^p(G/G_{(p+1)},\Ok) \to \Ext^p(\Ok/J,\Ok)
        \end{equation}
        given by $(g,\xi) \mapsto (\epsilon_g)_* \xi$ is non-degenerate.

        \noindent 2) The morphism
        \begin{equation*}
            G/G_{(p+1)} \to \Hom(\Ext^p(G/G_{(p+1)},\Ok),\Ext^p(\Ok/J,\Ok))
        \end{equation*}
        induced by \eqref{eq:pairing-to-extI} is surjective if and only if \eqref{eq:local-duality-surjection1} is.

        \noindent 3) The morphism 
        \begin{equation*}
             \Ext^p(G/G_{(p+1)},\Ok) \to \Hom(G/G_{(p+1)},\Ext^p(\Ok/J,\Ok))
        \end{equation*}
        induced by \eqref{eq:pairing-to-extI} is surjective if and only if \eqref{eq:local-duality-surjection2} is. 
    \end{lma}

    \begin{proof}
        We note first that by Lemma~\ref{lma:hom-to-extI}, it is enough to prove the corresponding
        statements for the pairing
        \begin{equation} \label{eq:pairing-to-extII}
            G/G_{(p+1)} \times \Ext^p(G,\Ok) \to \Ext^p(\Ok/J,\Ok).
        \end{equation}
        
        From the long exact sequence associated to the short exact sequence,
        \begin{equation*}
            0 \to G_{(p+1)} \to G \to G/G_{(p+1)} \to 0,
        \end{equation*}
        and using that $\Ext^k(G_{(p+1)},\Ok) = 0$ for $k=p-1,p$ by Proposition~\ref{prop:ext-vanishing},
        we obtain that the surjection $G \to G/G_{(p+1)}$ induces an isomorphism
        \begin{equation} \label{eq:extp-pure-isom}
            \Ext^p(G/G_{(p+1)},\Ok) \stackrel{\cong}{\to} \Ext^p(G,\Ok),
        \end{equation}
        so the pairing \eqref{eq:pairing-to-extII} is then isomorphic to the pairing
        \eqref{eq:pairing-to-extI}.
    \end{proof}

    \begin{prop} \label{prop:pure-nondeg-surjective}
        Assume that $G$ has pure codimension $p$, and let $J \subseteq \ann G$ be of pure codimension $p$.
        Consider the pairing
        \begin{equation} \label{pairing-to-ext-III}
            G \times \Ext^p(G,\Ok) \to \Ext^p(\Ok/J,\Ok)
        \end{equation}
        given by $(g,\xi) \mapsto (\epsilon_g)_* \xi$. 

        \noindent 1) The morphism
        \begin{equation} \label{eq:induced-simpler-I}
            \Ext^p(G,\Ok) \to \Hom(G,\Ext^p(\Ok/J,\Ok))
        \end{equation}
        induced by \eqref{pairing-to-ext-III} is always an isomorphism.

        \noindent 2) The morphism
        \begin{equation} \label{eq:induced-simpler-II}
            G \to \Hom(\Ext^p(G,\Ok),\Ext^p(\Ok/J,\Ok))
        \end{equation}
        induced by \eqref{pairing-to-ext-III} is injective, and is surjective if and only if $G$ is $S_2$.
    \end{prop}

\begin{proof}
    That \eqref{eq:induced-simpler-I} is an isomorphism follows by Corollary~\ref{cor:exthomp}.

    The morphism \eqref{eq:induced-simpler-II} factors as the composition of \eqref{eq:dualizing} with \eqref{eq:Psi}, i.e.,
    \begin{equation} \label{eq:local-duality-isomorphism2}
        G \stackrel{\cong}{\to} H^0(\Hom(\Hom(G,I_J),I_J)) \stackrel{\Psi}{\to} \Hom(\Ext^p(G,\Ok),\Ext^p(\Ok/J,\Ok)),
    \end{equation}
    where the first morphism is an isomorphism by Proposition~\ref{prop:dualizing}.
    By Proposition~\ref{prop:ssreplacement}, $\Psi$ is always injective, and is surjective if $G$ is $S_2$.

    If we now consider \eqref{eq:induced-simpler-I}, with $G$ replaced by $\Ext^p(G,\Ok)$ (which also has pure codimension $p$ by Proposition~\ref{prop:ext-lc-pure}),
    we have thus just proven that
    \begin{equation*}
        \Ext^p(\Ext^p(G,\Ok),\Ok) \cong \Hom(\Ext^p(G,\Ok),\Ext^p(\Ok/J,\Ok)),
    \end{equation*}
    where the left-hand side is $S_2$ by  Corollary~\ref{cor:extisS2},
    and hence, \eqref{eq:induced-simpler-II} can be an isomorphism only if $G$ is $S_2$.
\end{proof}

\subsection{Proofs of non-degeneracy and surjectivity of induced morphisms in Theorem~\ref{thm:local-duality-general}} \label{ssect:proof-general2}

\begin{lma} \label{lma:ses-extp}
    The inclusion $G_{(p)} \to G$ induces a short exact sequence
    \begin{equation} \label{eq:ses-extp}
        0 \to \Ext^p(G,\Ok)^{(p)} \to \Ext^p(G_{(p)},\Ok) \to H \to 0,
    \end{equation}
    where $H$ has codimension $\geq p+2$.
\end{lma}

\begin{proof}
Consider the short exact sequence
\begin{equation*}
    0 \to G_{(p)} \to G \to H' \to 0,
\end{equation*}
where $H' := G/G_{(p)}$. From the long exact sequence of $\Ext$, we have an exact sequence
\begin{equation*}
    \Ext^p(H',\Ok) \to \Ext^p(G,\Ok) \to \Ext^p(G_{(p)},\Ok) \to \Ext^{p+1}(H',\Ok),
\end{equation*}
and since $H'$ has only associated primes of codimension $<p$, $\Ext^p(H',\Ok)$ has codimension $\geq p+1$ by Proposition~\ref{prop:ext-ass-primes},
so we get a short exact sequence \eqref{eq:ses-extp}, where $H := \im (\Ext^p(G_{(p)},\Ok) \to \Ext^{p+1}(H',\Ok))$.
By Proposition~\ref{prop:ext-ass-primes}, $\Ext^{p+1}(H',\Ok)$ has codimension $\geq p+2$, since $H'$ has only associated primes of codimension $<p$,
and thus $H$, being a submodule of $\Ext^{p+1}(H',\Ok)$, has codimension $\geq p+2$.
\end{proof}

\begin{prop} \label{prop:reduction-codim-p}
    Let $G$ and $Z$ be as in Theorem~\ref{thm:local-duality-general}. Then the pairing
    \eqref{eq:general-duality-pairing-descended} is non-degenerate, and the induced morphisms
    \eqref{eq:general-duality-surjection1} and \eqref{eq:general-duality-surjection2} are
    surjective if and only if $G^{(p)}$ and $\Ext^p(G,\Ok)^{(p)}$ are $S_2$, respectively.
\end{prop}

\begin{proof}
    As in the proof of Lemma~\ref{lma:reduction-to-ext-pure}, it is enough to prove the corresponding
    statements for the pairing 
    \begin{equation*}
        G^{(p)} \times \Ext^p(G,\Ok)^{(p)} \to \Ext^p(\Ok/J,\Ok),
    \end{equation*}
    where $J \subseteq \ann G_{(p)}$.

    We first consider the morphism
    \begin{equation} \label{eq:local-duality-general-isomorphism1}
        \Ext^p(G,\Ok)^{(p)} \to \Hom(G^{(p)},\Ext^p(\Ok/J,\Ok))
    \end{equation}
    induced by the pairing. By Lemma~\ref{lma:ses-extp}, this factors as
    \begin{equation*}
        \Ext^p(G,\Ok)^{(p)} \to \Ext^p(G_{(p)},\Ok) \to \Hom(G^{(p)},\Ext^p(\Ok/J,\Ok)),
    \end{equation*}
    where the first morphism is injective, and by Proposition~\ref{prop:ext-ass-primes}, Corollary~\ref{cor:extisS2} and Corollary~\ref{cor:S2-isom},
    it is surjective if and only if $\Ext^p(G,\Ok)^{(p)}$ is $S_2$.
    By Proposition~\ref{prop:pure-nondeg-surjective} and \eqref{eq:extp-pure-isom}, the second morphism is an isomorphism. 

    It remains to consider the morphism 
    \begin{equation} \label{eq:local-duality-general-isomorphism2}
        G^{(p)} \to \Hom(\Ext^p(G,\Ok)^{(p)},\Ext^p(\Ok/J,\Ok)).
    \end{equation}
    By the long exact sequence of $\Ext$ associated to the short exact sequence \eqref{eq:ses-extp},
    and the fact that $\Hom(H,\Ext^p(\Ok/J,\Ok)) = 0$ since $H$ has codimension $\geq p+2$, and
    $\Ext^p(\Ok/J,\Ok)$ has pure codimension $p$, we obtain an exact sequence 
    \begin{align*}
        0 \to &\Hom(\Ext^p(G_{(p)},\Ok),\Ext^p(\Ok/J,\Ok)) \to  \\
        \to &\Hom(\Ext^p(G,\Ok)^{(p)},\Ext^p(\Ok/J,\Ok)) \to  
        \Ext^1(H,\Ext^p(\Ok/J,\Ok)).
    \end{align*}
    Since $\codim H \geq p+2$, and $\Ext^p(\Ok/J,\Ok)$ has pure codimension $p$, $$\depth(\ann H,\Ext^p(\Ok/J,\Ok)) \geq 2,$$
    and hence, $\Ext^1(H,\Ext^p(\Ok/J,\Ok)) = 0$ by Proposition~\ref{prop:ext-vanishing}.
    Hence, the first morphism is an isomorphism, and using this in combination with \eqref{eq:extp-pure-isom},
    we can reduce to the case $G = G^{(p)}$, and we then obtain the statement about surjectivity and
    injectivity of \eqref{eq:local-duality-general-isomorphism2} by Proposition~\ref{prop:pure-nondeg-surjective}.
\end{proof}

\section{The Andersson-Wulcan pairing} \label{sect:aw-pairing}

In this section, we give direct analytic expressions for the pairings \eqref{eq:local-duality-pairing} and \eqref{eq:general-duality-pairing}
with the help of residue currents, when $H^p_Z(\Ok)$ is represented as currents as in \eqref{eq:lc-current-repr}.

\subsection{Coleff-Herrera currents}

By \eqref{eq:lc-current-repr}, we can represent elements in $H^p_Z(\Ok)$
as $\dbar$-closed $(0,p)$-currents with support in $Z$ modulo $\dbar$
of such $(0,p-1)$-currents. When discussing cohomological residues below,
it will be useful that there is in fact a canonical choice of representative
in each such cohomology class. In fact, we also get just as in \cite{AndNoeth}
that our second way of defining the pairing indeed gives directly this
representative, see Remark~\ref{rem:chz-pairing}.

So called Coleff-Herrera currents
were introduced in \cite{DS1} (under the name ``locally residual currents''),
as canonical representatives of certain local cohomology classes.
Let $(Z,0)$ be the germ of a subvariety of $(\Cn,0)$ of pure codimension $p$.
A $(*,p)$-current $\mu$ on $(\Cn,0)$ is a \emph{Coleff-Herrera current},
denoted $\mu \in CH_Z$, if $\dbar \mu = 0$, $\overline{\psi} \mu = 0$ for
all holomorphic functions $\psi$ vanishing on $Z$, and $\mu$ has the
\emph{standard extension property}, SEP, with respect to $Z$.
We say that $\mu$ has the SEP if the limit $\lim_{\epsilon \to 0^+} \chi(|h|^2/\epsilon) \mu$ exists and is equal to $\mu$
for any tuple of holomorphic functions $h$ such that $\{ h = 0 \} \cap Z$ has
codimension $> p$, where $\chi(t) : \R \to \R$ is a smooth cut-off function
which is identically $0$ for $t$ close to $0$ and which is identically $1$
for $t \geq 1$.

This description of Coleff-Herrera currents is due to Bj\"ork, see
\cite[Chapter~3]{BjDmod}, and \cite[Section~6.2]{BjAbel}, . In \cite{DS1},
locally residual currents on $Z$ were defined as currents of the form
$\mu = \omega \wedge R$, with support on $Z$, where $\omega$ is a holomorphic
$(*,0)$-form, and $R$ is a Coleff-Herrera product of a tuple $(f_1,\dots,f_p)$
defining a complete intersection ideal $I = J(f_1,\dots,f_p)$ of codimension $p$.

As mentioned above, one of the main objectives of \cite{DS1}, and later
refined in \cite{DS2}, was to obtain canonical representatives of local
cohomology classes in $H^p_Z(\Ok)$ in terms of Coleff-Herrera products.
By \eqref{eq:lc-current-repr}, $H^p_Z(\Ok)$ is canonically isomorphic
to $H^p(C_Z^{0,\bullet})$. By Theorem~5.1 in \cite{DS2},
\begin{equation*}
   \ker(C_Z^{0,p} \stackrel{\dbar}{\to} C_Z^{0,p+1})
    = CH_Z \oplus \dbar C_Z^{0,p-1},
\end{equation*}
and one thus obtains an isomorphism
\begin{equation} \label{eq:ch-loc-coh}
    CH_Z  \stackrel{\cong}{\to} H^p(C^{0,\bullet}_Z) \cong H^p_Z(\Ok),
\end{equation}
so each element in $H^p_Z(\Ok)$ has a unique representative as a current
$\mu \in CH_Z$.

\subsection{Residue currents of Andersson-Wulcan} \label{ssect:aw-currents}

The duality theorem for Coleff-Herrera products, \eqref{eq:ch-duality}, was generalized by Andersson and Wulcan in \cite{AW1}.
Let $(E,\varphi)$ be a free resolution of length $N$ of a finitely generated $\Ok$-module $G$ of codimension $p > 0$,
such that $\varphi_1 : E_1 \to E_0$ is generically surjective,
and assume that $E_0,\dots,E_N$ are equipped with Hermitian metrics.
Andersson and Wulcan constructed in \cite{AW1} an associated $\Hom(E_0,E)$-valued residue current $R^E$
satisfying the following duality principle \cite{AW1}*{Theorem~1.1}: if
$g_0 \in E_0$, then
\begin{equation} \label{eq:aw-duality}
    R^E g_0 = 0 \text{ if and only if } g_0 \in \im \varphi_1.
\end{equation}
The current $R^E$ can be decomposed in the form
\begin{equation} \label{eq:Rparts}
    R^E = \sum_{k=p}^N R_k^E.
\end{equation}
where $R_k^E$ is a $\Hom(E_0,E_k)$-valued $(0,k)$-current. These currents satisfy that
\begin{equation} \label{eq:nablaR}
    \varphi_k R^E_k = \dbar R^E_{k-1},
\end{equation}
see \cite{AW1}*{Proposition~2.2}.

Let $G = \Ok/I$, where $I = (f_1,\dots,f_p)$ is a complete intersection ideal of codimension $p$,
and let $(E,\varphi)$ be the Koszul complex of $(f_1,\dots,f_p)$. Then it is not only the case that the
Coleff-Herrera product of $f$ and $R^E$ have the same annihilator, i.e., $I$, but they do in fact coincide.
More precisely, if we let $e_1,\dots,e_p$ be a trivial frame of $K_1 = \Ok^p$ such that the differential in
the Koszul complex is contraction with $\sum f_i e_i^*$, and $e_1\wedge \dots \wedge e_{p}$ is the induced frame
on $K_p$, then
\begin{equation} \label{eq:coincide-aw-ch}
    R^E = R^E_p = \dbar \frac{1}{f_p}\wedge \dots \wedge \dbar \frac{1}{f_1} \wedge e_1 \wedge \dots \wedge e_p,
\end{equation}
see \cite{AndCH}*{Corollary~3.2} and \cite{PTY}*{Theorem~4.1}.
Thus, the duality principle of Andersson-Wulcan is a direct generalization
of the duality theorem for Coleff-Herrera products.

As introduced in \cite{AW2}, a current of the form
\begin{equation*}
        \frac{1}{z_{i_1}^{n_1}}\cdots\frac{1}{z_{i_k}^{n_k}}\dbar\frac{1}{z_{i_{k+1}}^{n_{k+1}}}\wedge\cdots\wedge\dbar\frac{1}{z_{i_m}^{n_m}}\wedge \omega,
\end{equation*}
in some local coordinate system $z$, where $\omega$ is a smooth form with compact support is said to be an
\emph{elementary current}, and a current on a complex manifold is said to be
\emph{pseudomeromorphic}, if it can be written as a locally
finite sum of push-forwards of elementary currents under compositions of modifications
and open inclusions.
As can be seen from the construction, the Andersson-Wulcan currents $R^E_k$ are pseudomeromorphic.

An important property of pseudomeromorphic currents is that they satisfy the following
\emph{dimension principle},  \cite{AW2}*{Corollary~2.4}.
\begin{prop} \label{proppmdim}
    If $T$ is a pseudomeromorphic $(*,q)$-current with support on a variety $Z$,
    and $\codim Z > q$, then $T = 0$.
\end{prop}
This is a variant for pseudomeromorphic currents of the SEP, which we described
above for currents in $CH_Z$.

\subsection{A comparison formula for residue currents}

The following is a generalization of the transformation law for Coleff-Herrera products
to Andersson-Wulcan currents, and which is expressed with the help of the comparison
morphism as in Proposition~\ref{prop:comparison-morphism}.
We will use the following somewhat simplified version of \cite{LarComp}*{Theorem~3.2}.
The last part in the statement of the theorem is part of \cite{LarComp}*{Corollary~3.6}.

\begin{thm} \label{thm:Rcomparison}
Let $F$ and $G$ be finitely generated $\Ok$-modules and let $(E,\varphi)$ and $(K,\psi)$
be free resolutions of $G$ and $F$. Let $\alpha : F \to G$ be a morphism,
and let $a : (K,\psi) \to (E,\varphi)$ be a morphism of complexes, extending $\alpha$
as in Proposition~\ref{prop:comparison-morphism}.
Then, there exist pseudomeromorphic $\Hom(K_0,E_k)$-valued $(0,k-1)$-currents $M_k$ such that
\begin{equation} \label{eqRcomparison2}
    R^E_p a_0 - a_p R^K_p = \varphi_{p+1} M_{p+1} - \dbar M_p.
\end{equation}
For any $F$ and $G$, $M_p \psi_1 = 0$, and if $F$ and $G$ have codimension $\geq p$, then $M_p = 0$.
\end{thm}

One basic special case of Theorem~\ref{thm:Rcomparison} is when $F = \Ok/I$ and $G = \Ok/J$,
where $I = J(f_1,\ldots,f_p)$ and $J = J(g_1,\ldots,g_p)$ are complete intersection ideals of codimension $p$,
and $I \subseteq J$, or equivalently, there exists a holomorphic $(p\times p)$-matrix $A$ such that $(f_1,\ldots,f_p) = (g_1,\ldots,g_p) A$.
By Example~\ref{ex:ci-morphism}, if we take $(E,\varphi)$ and $(K,\psi)$ to be the Koszul complexes of $g$ and $f$ respectively,
then $a_p = \det A$. In addition, by \eqref{eq:coincide-aw-ch}, $R^E$ and $R^K$ are the Coleff-Herrera products of $g$ and $f$ respectively.
In addition, since $E_{p+1} = 0$, $M_{p+1} = 0$, since it is $\Hom(E_0,E_{p+1})$-valued,
and since $Z(I)$ and $Z(J)$ have codimension $p$, $M_p = 0$. Thus, in this case, the comparison formula,
\eqref{eqRcomparison2}, becomes the transformation law for Coleff-Herrera
products, \eqref{eq:transformation-law}.

\subsection{Definitions and properties of the pairing}

We here give an alternative expression for our pairing. This pairing is defined with the help of residue
current of Andersson and Wulcan from \cite{AW1}. The pairing appears explicitly in \cite{AndNoeth}, in the
case when $G$ has pure codimension $p$, including a proof that it then is non-degenerate, but it is not proven
there that the pairing is functorial in $G$.

If $(E,\varphi)$ is a free resolution of $G$, then the associated residue current $R^E_p$ as in \eqref{eq:Rparts}
takes values in $\Hom(E_0,E_p)$.
We use the identification of $\Ext^p(G,\Ok)$ with $H^p(\Hom(E_\bullet, \Ok))$, and one can thus represent
an element $\xi \in \Ext^p(G,\Ok)$ as
\begin{equation} \label{eq:xirepr}
    \xi = [\xi_0], \text{ where $\xi_0 \in \Hom(E_p,\Ok)$ is such that $\varphi_{p+1}^* \xi_0 = 0$}.
\end{equation}
In addition, for $g \in G_{(p)}$, we choose a representative
\begin{equation} \label{eq:grepr}
    g_0 \in E_0 \text{ such that } g = \pi(g_0) \text{ where } \pi : E_0 \to \coker \varphi_1 \cong G
\end{equation}
is the natural surjection.

The \emph{Andersson-Wulcan pairing} is defined by
\begin{equation} \label{eq:pairing-aw}
    \la g, \xi\ra_{AW} := \xi_0 R^E_p g_0,
\end{equation}
where $\xi_0$ and $g_0$ are as in \eqref{eq:xirepr} and \eqref{eq:grepr}.
Here, $\xi_0 R^E_p g_0$ is considered as an element of $H^p_Z(\Ok)$ using the representation \eqref{eq:lc-current-repr},
i.e., we claim that $\xi_0 R^E_p g_0$ is a $\dbar$-closed current in $C_Z^{0,p}$.
The fact that it is a $\dbar$-closed $(0,p)$-current follows exactly as in \cite{AndNoeth}. That it has its support
on $Z$ follows from that outside of $\supp g \subseteq \supp G_{(p)} \subseteq Z$, $g_0 \in \im \varphi_1$, and
hence, $\supp \xi_0 R^E_p g_0 \subseteq \supp g$ by \eqref{eq:aw-duality}.

\begin{lma} \label{lma:functorial-aw}
    The Andersson-Wulcan pairing \eqref{eq:pairing-aw} is well-defined,
    and functorial in $G$. More precisely, if $\alpha : F \to G$ is a morphism,
    $(K,\psi)$ and $(E,\varphi)$ are free resolutions of $F$ and $G$ respectively,
    and $a : (K,\psi) \to (E,\varphi)$ is a morphism of complexes extending $\alpha$,
    and $a_0 f_0 = g_0$, where $f_0 \in K_0$ and $g_0 \in E_0$ are representatives of
    $f \in F_{(p)}$ and $g \in G_{(p)}$, then
    \begin{equation} \label{eq:functorial-currents}
        \xi_0 R^E_p g_0 = a_p^* \xi_0 R^K_p f_0
    \end{equation}
    as currents.
\end{lma}

Note that in order to prove functoriality, it would be enough to prove that \eqref{eq:functorial-currents}
holds in \eqref{eq:lc-current-repr}, i.e., as cohomology classes, so that \eqref{eq:functorial-currents} holds
modulo $\dbar$ of a current vanishing on $Z$.

\begin{proof}
    The fact that \eqref{eq:pairing-aw} is independent of the choice of representative $g_0$ follows just as
    in \cite{AndNoeth}*{Theorem~1.2}. That it is independent of the choice of $\xi_0$ is proven as follows. If
    $\xi_1$ is another representative of $[\xi_0]$, then $\xi_0$ and $\xi_1$ differ by something in $\im \varphi_p^*$. Thus, it
    suffices to prove that
    \begin{equation*}
        \varphi_p R^E_p g_0 = 0.
    \end{equation*}
    By \eqref{eq:nablaR}, $\varphi_p R^E_p = \dbar R^E_{p-1}$, so using this, and the fact that $g_0$ is $\dbar$-closed,
    \begin{equation*}
        \varphi_p R^E_p g_0 = \dbar ( R^E_{p-1} g_0 ).
    \end{equation*}
    Since $g \in G_{(p)}$, $g = 0$ outside of $\supp G_{(p)}$ which has codimension $\geq p$,
    so $\supp R^E_{p-1} g_0$ is contained in a variety of codimension $\geq p$, and since it is a pseudomeromorphic
    $(0,p-1)$-current, it is $0$ by the dimension principle, Proposition~\ref{proppmdim}, so $\varphi_p R^E_p g_0 = 0$.

    Regarding independence of the choice of $(E,\varphi)$ with Hermitian metrics on $E$, it is
    a special case of functoriality in \eqref{eq:functorial-currents}, which we prove below, when we take
    $a : (K,\psi) \to (E,\varphi)$ to be the morphism induced by the identity map on
    $G$, if $(K,\psi)$ and $(E,\varphi)$ are both free resolutions of $G$.

    We now prove \eqref{eq:functorial-currents}.
    Note that if $f \in F$ and $f_0 \in K_0$ satisfies $\pi(f_0) = f$ (cf., \eqref{eq:grepr}),
    then $\pi( a_0 f_0 ) = \alpha f$. Thus, by the comparison formula, \eqref{eqRcomparison2},
    \begin{align*}
        &\la \alpha f , \xi \ra_{AW} = \xi_0 R^E_p a_0 f_0  = \xi_0 a_p R^K_p f_0 + \xi_0 \varphi_{p+1} M_{p+1} f_0 - \dbar (\xi_0 M_p f_0) = \\
        = & \la f,\alpha^* \xi \ra_{AW} + \xi_0 \varphi_{p+1} M_{p+1} f_0 - \dbar( \xi_0 M_p f_0),
    \end{align*}
    and it remains to prove that the last two terms vanish.
    We get that $\xi_0 \varphi_{p+1} M_{p+1} f_0 = 0$ since $\varphi_{p+1}^* \xi_0 = 0$. In addition, outside of
    $\supp F_{(p)}$, $f_0 = \psi_1 f_1$ for some $f_1$, so $\supp M_p f_0 \subseteq \supp F_{(p)}$, which has codimension $p$.
    Thus, $M_p f_0 = 0$, since its support has codimension $\geq p$, and it is a pseudomeromorphic $(0,p-1)$-current,
    so it is $0$ by the dimension principle, Proposition~\ref{proppmdim}.
\end{proof}

\begin{ex} \label{ex:aw-pairing-ci}
    We now consider the most basic case when $G = \Ok/I$, where $I$ is a complete intersection ideal
    of codimension $p$, given as $I = J(f_1,\ldots,f_p)$. Then, the Koszul complex $(K,\psi)$ of $f$
    is a free resolution of $\Ok/I$. Using the representation \eqref{eq:ext-ci-isom-explicit} of $\Ext^p(\Ok/I,\Ok)$
    together with the expression \eqref{eq:coincide-aw-ch} for the current $R^K_p$, we get that the pairing
    \eqref{eq:pairing-aw} in this case becomes
    \begin{equation*}
        \la g, h(e_1 \wedge \cdots \wedge e_p)^* \ra_{AW} = g h \dbar\frac{1}{f_p} \wedge \cdots \wedge \dbar \frac{1}{f_1}.
    \end{equation*}
    In this case, by \eqref{eq:pairing-ch}, the pairing \eqref{eq:pairing-grothendieck} coincides indeed
    with the pairing \eqref{eq:pairing-aw} using the canonical isomorphism \eqref{eq:ext-ci-isom-explicit}
    between the different representations of $\Ext^p(\Ok/I,\Ok)$ being used to define the pairings.
\end{ex}

\begin{prop} \label{prop:aw-gr-coincide}
    Let $G$ be a finitely generated $\Ok$-module, and let $Z \subseteq (\C^n,0)$ be a subvariety of pure codimension $p$.
    Under the canonical isomorphism between the representations of $H^p_Z(\Ok)$ induced by the isomorphism \eqref{eq:extrepr},
    the pairings \eqref{eq:pairing-grothendieck} and \eqref{eq:pairing-aw} coincide.
\end{prop}

\begin{proof}
    It follows from Examples~\ref{ex:aw-pairing-ci} that this holds when $G = \Ok/I$, where $I$ is a complete intersection
    ideal of codimension $p$. By functoriality, and taking the morphism $\alpha : (\Ok/I)^r \to G$ from
    Lemma~\ref{lma:complete-intersection-surjective}, which is surjective onto $G_{(p)}$, it follows from
    functoriality of both pairings that they coincide also in general.

    Alternatively, one can prove the proposition in the following way. In \cite{AndNoeth}*{Theorem~1.5},
    Andersson proves the following generalization of \eqref{eq:extrepr2}, that for $G$ a finitely generated $\Ok$-module
    of pure codimension $p$ and $(E,\varphi)$ a free resolution of $G$, the canonical isomorphism
    \begin{equation} \label{eq:extrepr2-gen}
        H^p(\Hom(E_\bullet,\Ok)) \cong H^p(\Hom(G,C^{0,\bullet}))
    \end{equation}
    is given by
    \begin{equation} \label{eq:extrepr3-gen}
        \xi = [\xi_0] \mapsto \xi_0 R^E_p.
    \end{equation}
    Here, $R^E_p$ is considered as $\Hom(G,C^{0,\bullet})$-valued by $R^E_p(g) := R^E_p g_0$, where
    $g_0$ is a representative of $g \in G$, as in \eqref{eq:grepr}.
    We now let $(K,\psi)$ be a free resolution of $\Ok/J$, and consider the following commutative diagram
    coming from the morphism $\epsilon_g : \Ok/J \to G$, which induces morphisms $\epsilon_g^* : \Ext^p(G,\Ok) \to \Ext^p(\Ok/J,\Ok)$, and the canonical
    isomorphisms \eqref{eq:extrepr2} and \eqref{eq:extrepr2-gen} between the different representations of $\Ext$.
    \begin{equation*}
\xymatrix{
    H^p(\Hom(E_\bullet,\Ok)) \ar[r]^{\cong} \ar[d]^{\epsilon_g^*} & H^p(\Hom(G,C^{0,\bullet})) \ar[d]^{\epsilon_g^*} \\
    H^p(\Hom(K_\bullet,\Ok)) \ar[r]^{\cong}        & H^p(\Hom(\Ok/J,C^{0,\bullet})).
}
    \end{equation*}
    The map $\epsilon_g^*$ on the right is just precomposition with the map $\epsilon_g$, so considering the element
    $\xi = [\xi_0] \in H^p(\Hom(E_\bullet,\Ok))$, this will be mapped in the diagram as follows:
    \begin{equation*}
\xymatrix{
    [\xi_0]  \ar[r] \ar[d] & [\xi_0 R^E_p] \ar[d]  \\
    \epsilon_g^* [\xi_0]  \ar[r]       & [\xi_0 R^E_p g_0].
}
    \end{equation*}
    Finally, using the representation \eqref{eq:lc-current-repr}, the map
    $\pi_1 : \Hom(\Ok/J,C^{0,\bullet}) \to H^p_Z(\Ok)$ acts just as the identity on currents,
    i.e.,
    \begin{equation*}
        \pi_1 [\xi_0 R^E_p g_0] = [\xi_0 R^E_p g_0] = \la \xi, g\ra_{AW} \in H^p_Z(\Ok).
    \end{equation*}
\end{proof}

\begin{ex} \label{ex:ci-reduction}
    As we do in Section~\ref{sect:linkage}, and as well in the proof of the preceding lemma,
    we can reduce the expression of the pairing to the complete intersection case, which when
    expressed in terms of currents thus becomes an expression in terms of Coleff-Herrera products.
    We take the map $\alpha : (\Ok/I)^r \to G$ as in Lemma~\ref{lma:complete-intersection-surjective},
    and for $g \in G_{(p)}$, we take $h \in (\Ok/I)^r$ such that $\alpha(h) = g$, which is possible
    since $\alpha$ is surjective onto $G_{(p)}$.
    Then, by functoriality of the pairing and Example~\ref{ex:aw-pairing-ci},
\begin{equation} \label{eq:pairing-lj}
    \la g, \xi \ra_{AW} = \dbar \frac{1}{f_p} \wedge \dots \wedge \dbar \frac{1}{f_1} \wedge \xi_0 a_p( h e_1 \wedge \dots \wedge e_p),
\end{equation}
    where $\xi_0$ is as in \eqref{eq:xirepr}.
\end{ex}

\begin{ex} \label{ex:artinian}
    In the special case above, when $G$ is \emph{Artinian}, i.e., when $\supp G = \{ 0 \}$, then by the Nullstellensatz,
    one can always choose the complete intersection ideal $I$ in Lemma~\ref{lma:complete-intersection-surjective} to be of
    the form $I = J(z_1^{N_1},\dots,z_n^{N_n})$, and one thus gets a representation
\begin{equation} \label{eq:pairing-lj-art}
    \la g, \xi \ra_{AW} = \dbar \frac{1}{z_1^{N_1}} \wedge \dots \wedge \dbar \frac{1}{z_n^{N_n}} \wedge \xi_0 a_n( h e_1 \wedge \dots \wedge e_n).
\end{equation}
\end{ex}

\begin{remark}
    The equality \eqref{eq:pairing-lj} was one of the starting points of this article.
    If we in particular consider the Artinian case as in Example~\ref{ex:artinian}, then the existence
    of the pairing \eqref{eq:pairing-lj-art} is elementary, using only the Nullstellensatz
    and the syzygy theorem, while the existence of the currents in \eqref{eq:pairing-aw} is rather non-elementary
    also in this special situation (without using the comparison formula).
    One of the starting aims in writing this article was then to try to find also a more elementary proof of the
    non-degeneracy of the pairing defined by the right-hand side of \eqref{eq:pairing-lj-art}, which is then achieved
    by the elementary proofs of non-degeneracy in the case when $f = (z_1^{N_1},\ldots,z_n^{N_n})$, combined with
    Lemma~\ref{lma:reduction-second}, which reduces non-degeneracy in the second argument to this case,
    and Lemma~\ref{lma:reduction-first}, which with the help of the theory of linkage reduces non-degeneracy
    in the first argument to this case.
\end{remark}

\begin{prop} \label{prop:nondegenerate-aw}
    If the pairing \eqref{eq:general-duality-pairing} is given as $\la \bullet, \bullet \ra_{AW}$,
    as defined in \eqref{eq:pairing-aw}, then the induced pairing \eqref{eq:general-duality-pairing-descended}
    is non-degenerate.
\end{prop}

\begin{proof}
    For $G = \Ok/I$, where $I$ is a complete intersection ideal, non-degeneracy in both arguments follows from the duality
    theorem for Coleff-Herrera products, as explained in Example~\ref{ex:aw-pairing-ci}.
    Thus, by Lemma~\ref{lma:reduction-second}, the pairing is non-degenerate in the second argument.

    To prove non-degeneracy in the first argument, we assume first that we are outside of $Z_{p+1}$.
    Then $G$ has a free resolution of length $\leq p$. Since the pairing is independent of the free resolution,
    we can thus assume that the free resolution $(E,\varphi)$ of $G$ has length $\leq p$. If $g \in G_{(p)}$,
    then $g = 0$ outside of $\supp G_{(p)}$, so $R^E_k g_0 = 0$ outside of $\supp G_{(p)}$ which has codimension $\geq p$.
    Hence, $R^E_k g_0 = 0$ for $k < p$ since it is a pseudomeromorphic $(0,k)$-current with support on a subvariety of
    codimension $\geq p$, and thus is $0$ by the dimension principle, Proposition~\ref{proppmdim}.
    In addition, if $(E,\varphi)$ has length $\leq p$, then $\ker \varphi_{p+1}^* = E_p^*$, so if $\xi R^E_p g_0 = 0$
    for all $\xi \in \ker \varphi_{p+1}^*$, then $R^E_p g_0 = 0$. Thus, $R^E g_0 = 0$, so $g = 0$ by \eqref{eq:aw-duality}.
    To conclude, $g = 0$ outside of $Z_{p+1}$ which has codimension $\geq p+1$, by \eqref{eq:Ext-Zk} and Proposition~\ref{prop:ext-ass-primes},
    so $g \in G_{(p+1)}$.
\end{proof}

\begin{ex} \label{ex:artinian-monomial}
    The explicitness of the pairing \eqref{eq:pairing-lj-art} depends on the ability of calculating the
    morphism $a_n$. In \cite{LW1}, we show that if $J$ is an \emph{Artinian monomial ideal}, i.e., an Artinian
    ideal generated by monomials, then one can explicitly compute the morphism $a$ when $(E,\varphi)$ is the
    so-called Hull resolution of $\Ok/J$.

In \cite{LW1}, we then use the explicit expression of $a$ to express the current $R^E_n$, which thus
also give an explicit description of the pairing \eqref{eq:pairing-lj-art}. For example, when $J \subseteq \Ok_{\C^2_{z,w},0}$ is the ideal
$J = J(z^a,z^b w^c, w^d)$, where $b < a$ and $c < d$, then $\Ok/J$ has the Hull resolution
\begin{equation*}
    0 \to \Ok^2 \stackrel{\varphi_2}{\longrightarrow} \Ok^3 \stackrel{\varphi_1}{\longrightarrow} \Ok \to \Ok/J \to 0,
\end{equation*}
where
\begin{equation*}
    \varphi_2 = \left[ \begin{array}{cc} -w^c & 0 \\ z^{a-b} & -w^{d-c} \\ 0 & z^b \end{array} \right]
    \text{ and }
    \varphi_1 = \left[ \begin{array}{ccc} z^a & z^b w^c & w^d \end{array} \right],
\end{equation*}
and the current $R^E_2$ is
\[
    \def\arraystretch{2.2}
    R^E_2 = \left[ \begin{array}{c} \dbar \dfrac{1}{w^c} \wedge \dbar \dfrac{1}{z^a} \\ \dbar \dfrac{1}{w^d} \wedge \dbar \dfrac{1}{z^b} \end{array} \right].
\]
If $\xi = [(\xi_1,\xi_2)^*] \in (\Ok^2)^*/(\im \varphi_2^*) \cong \Ext^2(\Ok/J,\Ok)$, then the pairing \eqref{eq:pairing-lj-art} is
\begin{equation*}
    \la g,\xi \ra_{AW} = \xi R^E_2 g = \xi_1 \dbar \dfrac{1}{w^c} \wedge \dbar \dfrac{1}{z^a} g + \xi_2 \dbar \dfrac{1}{w^d} \wedge \dbar \dfrac{1}{z^b} g.
\end{equation*}
The non-degeneracy of the pairing in the first argument thus corresponds to the decomposition of $J$,
\begin{equation*}
    J  = \ann \left( \dbar \dfrac{1}{w^c} \wedge \dbar \dfrac{1}{z^a}\right) \cap \ann \left( \dbar \dfrac{1}{w^d} \wedge \dbar \dfrac{1}{z^b}\right) = J(z^a,w^c) \cap J(z^b,w^d).
\end{equation*}
\end{ex}

\begin{remark} \label{rem:chz-pairing}
    We finally also notice the difference in formulation of
    Theorem~\ref{thm:local-duality} and the main theorem in \cite{AndNoeth}.
    By \eqref{eq:ch-loc-coh}, we could just as well formulate
    \eqref{eq:local-duality-pairing-descended} as that there exists a
    non-degenerate pairing
    \begin{equation}
        G/G_{(p+1)} \times \Ext^p(G,\Ok) \to CH_Z,
    \end{equation}
    which is indeed the formulation used in \cite{AndNoeth}.
    Note that just as in \cite{AndNoeth}, if we use \eqref{eq:pairing-aw}
    to define the pairing, then this gives directly the representative
    in $CH_Z$.

    Similarly, Theorem~\ref{thm:local-duality-surjective} and
    Theorem~\ref{thm:local-duality-general} can be reformulated as
    that there exists a canonical non-degenerate pairing
    \begin{equation} \label{eq:general-duality-pairing-ch}
        G_{(p)} \times \Ext^p(G,\Ok) \to CH_Z,
    \end{equation}
    such that the induced injective morphisms
    \begin{equation} \label{eq:general-duality-surjection-ch1}
        G^{(p)} \to \Hom(\Ext^p(G,\Ok)^{(p)},CH_Z)
    \end{equation}
    and
    \begin{equation} \label{eq:general-duality-surjection-ch2}
        \Ext^p(G,\Ok)^{(p)} \to \Hom(G^{(p)},CH_Z)
    \end{equation}
    are surjective if and only if $G^{(p)}$ and $\Ext^p(G,\Ok)^{(p)}$
    are $S_2$ respectively, and if $G$ has codimension $\geq p$,
    the latter is automatic, and one has an isomorphism
    \begin{equation} \label{eq:general-duality-surjection-ch-p}
        \Ext^p(G,\Ok) \cong \Hom(G^{(p)},CH_Z).
    \end{equation}
\end{remark}

\section{Cohomological residues} \label{sect:coh-residues}

In this section, we prove Theorem~\ref{thm:cohomological-duality}, and
describe the alternative description of the pairing \eqref{eq:cohomological-duality-pairing}
when $G$ has codimension $\geq p$.

\subsection{Cohomological residues of Lundqvist} \label{ssect:lund}

The description of the pairing \eqref{eq:cohomological-duality-pairing} when $G$ has codimension $\geq p$
is based on a construction by Lundqvist in \cite{Lund1} and \cite{Lund2}. In these articles, only the
case when $G$ is of the form $G = \Ok/J$, where $J$ has pure codimension $p$ is considered, but the
construction of the pairing works the same also in this more general setting.

Let $G$ be a finitely generated $\Ok$-module with a free resolution $(E,\varphi)$ of length $N$,
equipped with some Hermitian metrics, and let $Z = Z(\supp G)$. We let $\Omega$ be a neighbourhood
of $0$ such that the free resolution exists and is pointwise exact on $\Omega\setminus Z$.
By \cite{AW1} (cf., \cite{Lund2}*{Section~2}), there exist smooth $\Hom(E_0,E_k)$-valued $(0,k-1)$-forms
$u_k$ for $k=1,\dots,N$ and $\Hom(E_1,E_k)$-valued $(0,k-2)$-forms $u^1_k$ for $k=2,\dots,N$ on $\Omega \setminus Z$ such that
\begin{align} 
    \begin{gathered}\label{eq:lundqvist-udef}
    \varphi_1 u_1 = \Id_{E_0} \text{, }\quad \dbar u_N = 0,\quad \varphi_2 u^1_2 = u_1 \varphi_1 + \Id_{E_1} \\
    \text{ and }\quad \varphi_{k+1} u_{k+1}^1 = u_k \varphi_1 + \dbar u_k^1 \text{ for $k=2,\dots,N-1$.}
    \end{gathered}
\end{align}
In \cite{Lund2}, this is expressed in the more compact notation $\nabla_{{\rm End}(E)}u = \Id_{E}$.

The main theorem in \cite{Lund2}, Theorem~3.3, can be reformulated as saying that if $G$ is of the form $G = \Ok/J$,
where $J$ has pure codimension $p$, then
\begin{equation} \label{eq:lund-reformulation}
    g \in \Ok/J \text{ if and only if } \int \xi_0 u_p g \wedge \dbar \beta = 0,
\end{equation}
for all $\beta \in H^{n,n-p}_{Z^c}$ and $[\xi_0] \in H^p(\Hom(E_\bullet,\Ok)) \cong \Ext^p(\Ok/J,\Ok)$.
In the case when $G = \Ok/I$, where $I$ is a complete intersection ideal of codimension $p$, then $u_p$ coincides
with the form $B_f$ defined by Passare, as in the introduction, and thus, the result of Lundqvist is a generalization of the result
of Passare.

For future reference, we remark that the current $R^E_p$ in Section~\ref{ssect:aw-currents} is defined as
\begin{equation} \label{eq:RU}
    R^E_p = \varphi_{p+1} U_{p+1} - \dbar U_p,
\end{equation}
where $U_p$ and $U_{p+1}$ are currents on $\Omega$ which are the so-called standard extensions
of $u_p$ and $u_{p+1}$, which in particular means that $u_p$ and $u_{p+1}$ coincide with $U_p$ and $U_{p+1}$
where they are smooth.

\subsection{A definition of the pairing \eqref{eq:cohomological-duality-pairing}} \label{ssect:lund-def}

Even though the main result is not formulated in this way, the pairing still appears
in the construction of Lundqvist, see \cite{Lund2}*{(2) and (9)}.
Let $G$ be an finitely generated $\Ok$-module of codimension $\geq p$, and let $Z \supseteq \supp G$
be of pure codimension $p$.
The \emph{Lundqvist pairing}
\begin{equation*}
    G \times \Ext^p(G,\Ok) \to \Hom_\C(H^{n,n-p}_{Z^c},\C)
\end{equation*}
is defined by
\begin{equation} \label{eq:pairing-lu}
    \la g, \xi \ra_{Lu}(\beta) := \int \xi_0 u_p g_0 \wedge \dbar \beta,
\end{equation}
where $\xi_0$ and $g_0$ are as in \eqref{eq:xirepr} and \eqref{eq:grepr},
and $\beta \in H^{n,n-p}_{Z^c}$. It is implicitly assumed that
$\beta$ has small enough support such that $\xi_0$, $u_p$ and $g_0$ are all defined
on the support of $\beta$.

\begin{lma}
    The pairing \eqref{eq:pairing-lu} is well-defined, i.e., independent of the choice of representative $g_0 \in E_0$
    of $g$, and $\xi_0 \in \ker \varphi_{p+1}^*$ of $[\xi_0]$.
\end{lma}

\begin{proof}
    To see that this pairing indeed is well-defined, we note first that
    if $g_0 = \varphi_1 g_1$, then if $p > 1$, since $u_p \varphi_1 = \varphi_p u_{p-1} - \dbar u^1_p$ by \eqref{eq:lundqvist-udef},
    and $\xi_0 \varphi_p = 0$, we get that $\xi_0 u_p g_0 = \dbar(\xi_0 u^1_p g_1)$. Thus, by Stokes' theorem,
    \begin{equation*}
        \int \xi_0 u_p g_0 \wedge \dbar \beta = \int \dbar (\xi_0 u^1_p g_1) \wedge \dbar \beta = 0.
    \end{equation*}
    If $p = 1$, then $u_1 \varphi_1 = \varphi_2 u^1_2-\Id_{E_1}$, and since $\xi_0 \varphi_2 = 0$,
    \begin{equation*}
        \int \xi_0 u_1 g_0 \wedge \dbar \beta = -\int \xi_0 g_1 \dbar \beta = 0,
    \end{equation*}
    where the last equality holds by Stokes' theorem, since $\xi_0 g_1$ is holomorphic on $\supp \beta$.
    Thus, \eqref{eq:pairing-lu} is independent of the representative $g_0 \in E_0$.

    Similarly, if $\xi_0 = \varphi_p^* \xi_1$, then if $p > 1$, we have that $\varphi_p u_p = \dbar u_{p-1}$, so
    \begin{equation*}
        \int \xi_0 u_p g_0 \wedge \dbar \beta = \int \xi_1 \varphi_p u_p \wedge \dbar \beta
        = \int \dbar(\xi_1 u_{p-1}) \wedge \dbar \beta = 0,
    \end{equation*}
    where the last equality follows by Stokes' theorem.
    If $p = 1$, then $\varphi_1 u_1 = \Id_{E_0}$, so
    \begin{equation*}
        \int \xi_0 u_1 g_0 \dbar \beta = \int \xi_1 g_0 \dbar \beta = 0,
    \end{equation*}
    by Stokes' theorem, since $\xi_1 g_0$ is holomorphic on $\supp \beta$.
\end{proof}

\begin{prop} \label{prop:lu-functorial}
    The pairing \eqref{eq:pairing-lu} is functorial in $G$.
\end{prop}

By functoriality, we mean the diagram similar to \eqref{eq:functorial} is commutative for this pairing.
The functoriality then holds for $Z \supseteq (\supp F) \cup (\supp G)$.
However, in contrast to the pairing in \eqref{eq:local-duality-pairing} and \eqref{eq:general-duality-pairing},
where we have the injective map $H^p_Z(\Ok) \to H^p_W(\Ok)$ for $Z \subseteq W$, by Lemma~\ref{lma:HpZW},
it is not clear that we would have anything similar for the map $\Hom_\C(H^{n,n-p}_{Z^c},\C) \to \Hom_\C(H^{n,n-p}_{W^c},\C)$.

\begin{proof}
    We consider $\alpha : F \to G$, where $F$ and $G$ are finitely generated $\Ok$-modules of codimension $\geq p$.
    If $f \in F$, and $\xi \in \Ext^p(G,\Ok)$, then we want to prove that
    \begin{equation*}
        \la \alpha(f) , \xi \ra_{Lu} = \la f, \alpha^* \xi \ra_{Lu}.
    \end{equation*}
    We let $(K,\psi)$ and $(E,\varphi)$ be free resolutions of $F$ and $G$ respectively,
    and let $a : (K,\psi) \to (E,\varphi)$ be a morphism of complexes extending $\alpha : F \to G$
    as in Proposition~\ref{prop:comparison-morphism}. We also let $\pi$ denote both the natural surjections
    $\pi : K_0 \to \coker \psi_1 \cong F$ and $\pi : E_0 \to \coker \varphi_1 \cong G$.
    If $\pi( f_0 ) = f$, then $\pi(a_0 f_0 ) = \alpha f$.

    It follows from the proof of \cite{LarComp}*{Theorem~3.2}, that if one for $\ell = p$ and $\ell = p+1$,
    lets $M_\ell$ be the $\Hom(K_0,E_\ell)$-valued $(0,\ell-1)$-form
    \begin{equation*}
        M_\ell = \sum_{k=1}^{\ell-1} (u^E)_\ell^k a_k (u^F)_k^0,
    \end{equation*}
    which is smooth outside of $Z := (\supp G) \cup (\supp F)$, then
    \begin{equation*}
        u^E_p a_0 - a_p u^F_p = \dbar M_p - \varphi_{p+1} M_{p+1}
    \end{equation*}
    outside of $Z$. One then obtains that
    \begin{align*}
         &(\la \alpha f, \xi \ra_{Lu})(\beta) = \int \xi_0 u_p^E a_0 f_0 \wedge \dbar \beta = \\
        = &\int \xi_0 a_p u_p^F f_0 \wedge \dbar \beta + \int \xi_0 \dbar M_p f_0 \wedge \dbar \beta
        = (\la f, \alpha^* \xi \ra_{Lu})(\beta)
    \end{align*}
    where the term involving $\varphi_{p+1} M_{p+1}$ vanishes, since $\xi_0 \varphi_{p+1} = 0$, and
    the integral involving $\dbar M_p$ vanishes by Stokes' theorem, and in the last equality,
    we used that $\alpha^* \xi = [\xi_0 a_p]$.
\end{proof}

\begin{cor}
    The pairing \eqref{eq:pairing-lu} descends to a pairing
    \begin{equation*}
        G/G_{(p+1)} \times \Ext^p(G,\Ok) \to \Hom_\C(H^{n,n-p}_{Z^c},\C)
    \end{equation*}
\end{cor}

\begin{proof}
    By \eqref{eq:extp-pure-isom}, the surjection $\alpha : G \mapsto G/G_{(p+1)}$ induces
    an isomorphism $\Ext^p(G,\Ok) \cong \Ext^p(G/G^{(p+1)},\Ok)$.
    Thus, we can write $\xi \in \Ext^p(G,\Ok)$ as $\alpha^* \xi_1 \in \Ext^p(G/G_{(p+1)},\Ok)$.
    Thus, if $g \in G_{(p+1)}$, $\alpha(g) = 0$, so
    \begin{equation*}
        \la g, \xi \ra = \la g, \alpha^* \xi_1 \ra = \la \alpha(g), \xi_1 \ra = 0.
    \end{equation*}
\end{proof}

We now note the following consequence of the results of Lundqvist.

\begin{prop} \label{prop:nondeg-lu}
    Let $G = \Ok/J$, where $J$ has pure codimension $p$. Then the pairing \eqref{eq:pairing-lu}
    is non-degenerate in the first argument.
\end{prop}

\begin{proof}
    This follows directly from \eqref{eq:lund-reformulation}. In the proof in \cite{Lund1} and \cite{Lund2},
    this is formulated for $Z = Z(J)$, but from the proof, it is seen that this proof works just as well
    for any $Z \supseteq Z(J)$ of pure codimension $p$.

    Alternatively, just as in the proof of Lemma~\ref{lma:reduction-first}, we could use the theory of linkage,
    and functoriality of the pairing which is proven below, to reduce non-degeneracy to the complete intersection
    case from Passare, \eqref{eq:passare-pairing}.
\end{proof}

Thus, combining non-degeneracy in the first argument of \eqref{eq:pairing-lu} by Lundqvist,
and non-degeneracy in the second argument from Passare, \eqref{eq:passare-pairing}, and the variant of
Lemma~\ref{lma:reduction-second} in this setting, we obtain an independent proof of
Theorem~\ref{thm:cohomological-duality} when the pairing is given by \eqref{eq:pairing-lu}.

\begin{lma}
    If $G$ has codimension $\geq p$, then the pairing \eqref{eq:cohomological-duality-pairing}, defined
    as the composition of the pairing \eqref{eq:general-duality-pairing-descended} with \eqref{eq:local-cohomology-dual}
    coincides with the pairing \eqref{eq:pairing-lu}.
\end{lma}

\begin{proof}
    In order to prove this, we use the representation \eqref{eq:pairing-aw} of the pairing \eqref{eq:general-duality-pairing-descended}.
    Taking the image of this under \eqref{eq:local-cohomology-dual}, and letting it act on $\beta \in H^{n,n-p}_{Z^c}$,
    \begin{equation*}
        R(\la g, [\xi_0]\ra_{AW})(\beta) = \int \xi_0 R^E_p g_0 \wedge \beta.
    \end{equation*}
    By \eqref{eq:RU}, and the fact that $\xi_0 \varphi_{p+1} = 0$, we get that
    \begin{equation*}
        R(\la g, [\xi_0]\ra_{AW})(\beta) = \int \xi_0 \dbar U_p g_0 \wedge \beta,
    \end{equation*}
    and by Stokes, and the fact that $U_p = u_p$ on $\supp \dbar \beta$, where $u_p$ is smooth,
    we get that
    \begin{equation*}
        R(\la g, [\xi_0]\ra_{AW})(\beta) = \int \xi_0 u_p g_0 \wedge \dbar \beta = (\la g, [\xi_0]\ra_{Lu})(\beta).
    \end{equation*}
\end{proof}

\subsection{Proof of Theorem~\ref{thm:cohomological-duality}}

Note that non-degeneracy in Theorem~\ref{thm:cohomological-duality} gives non-degeneracy in Theorem~\ref{thm:local-duality},
so the results of Lundqvist, as discussed in the previous section give non-degeneracy in the first argument in
Theorem~\ref{thm:local-duality} for $G = \Ok/J$, where $J$ has pure codimension $p$. We now show that we can always
go the other way around as well.

\begin{lma} \label{lma:lc-injective}
    The map \eqref{eq:local-cohomology-dual} is injective.
\end{lma}

\begin{proof}
    By representing elements in $H^p_Z(\Ok)$ as Coleff-Herrera currents, see \eqref{eq:ch-loc-coh} and using that
    such currents have the SEP, it is enough to assume that we
    are on $Z_{\rm reg}$, and we then choose coordinates such that locally,
    \begin{equation*}
        Z_{\rm reg} = \{ w_1 = \dots = w_p = 0 \} \subseteq \C^{n-p}_z \times \C^p_w.
    \end{equation*}
    If we take $T \in H^p_Z(\Ok)$, and take its representative $\mu \in CH_Z$, then by \cite{AndCH}*{Lemma~3.6},
    we can write $\mu$ as
    \begin{equation*}
        \mu = \sum_{|\alpha| \leq M} a_\alpha(z) \dbar \frac{1}{w^\alpha}.
    \end{equation*}
    We then let $z \in Z_{\rm reg}$ be fixed, and denote the variables $(\zeta,w)$ on $\C^n$, and define the test-form
    \begin{equation*}
        \beta^\alpha_z := \chi(|w|) w^{\alpha-{\bf 1}} dw_1 \wedge \dots \wedge dw_p \wedge  \dbar \chi(|\zeta-z|) k_{BM}(\zeta-z),
    \end{equation*}
    where $\chi(t)$ is a cut-off function which is $\equiv 1$ for $t$ sufficiently close to $0$, and which
    is $\equiv 0$ for $t$ sufficiently large, and $k_{BM}$ is the Bochner-Martinelli kernel in $(n-p)$ variables.
    Since the part of $\beta^\alpha_z$ depending on $\zeta-z$ is $\dbar$-closed, $\dbar \beta^\alpha_z$ has support on $\supp \chi'(|w|) \cap \supp \chi'(|\zeta-z|)$,
    which does not intersect $Z_{\rm reg}$, so $\beta^\alpha_z \in H^{n,n-p}_{Z^c}$.

    If we want to show that the map is injective, we thus assume that
    \begin{equation*}
        \int \mu \wedge \beta^\alpha_z = 0,
    \end{equation*}
    which by the Bochner-Martinelli formula gives that
    \begin{equation*}
        a_\alpha(z) = 0,
    \end{equation*}
    and thus, the map is injective on $Z_{\rm reg}$.
\end{proof}

\begin{remark} \label{rem:ci-reduction}
    Just as for the previous pairings, if $G$ is a finitely generated $\Ok$-module, one can
    always find a complete intersection ideal $I \subseteq \ann G_{(p)}$ and a surjective morphism
    $\alpha : (\Ok/I)^r \to G_{(p)}$. By the functoriality, one then gets if $g = \pi(f)$ that
    \begin{equation*}
        \la g, \xi \ra(\beta) = \la f, \alpha^* \xi \ra(\beta)
    \end{equation*}
    for $\beta \in H^{n,n-p}_{Z(I)^c}$. Note that this does not work for any $\beta \in H^{n,n-p}_{Z^c}$,
    but just the ones which are $\dbar$-closed outside of $Z(I)$.

    For any $Z$ of codimension $p$, one can always find a complete intersection $W$ of codimension $p$
    containing $Z$, but it is not clear to us whether for $\beta \in H^{n,n-p}_{Z^c}$, one can always
    find a complete intersection $W \supseteq Z$ of codimension $p$ such that $W \cap \supp \dbar \beta = \emptyset$,
    and thus, it is not clear that one can always express the residue \eqref{eq:pairing-lu} in terms of the
    pairing for complete intersections as in Example~\ref{ex:ci-reduction}.
    When $G$ is Artinian however, this reduction always work, just as in Example~\ref{ex:artinian},
    as the following example shows.
\end{remark}

\begin{ex} \label{ex:lj}
    Using the notation from Example~\ref{ex:artinian}, and the expression \eqref{eq:pairing-pa}
    for the pairing in the complete intersection case, we obtain the following expression,
\begin{equation} \label{eq:pairing-ljp}
    \la g, \xi \ra_{Lu}(\beta) = \frac{1}{(2\pi i)^n} \int \xi_0 a_n( h e_1 \wedge \dots \wedge e_n) \wedge B \wedge \dbar \beta,
\end{equation}
    where
    \begin{equation*}
        B(z) = \frac{\sum (-1)^{k-1} \overline{z_k^{N_k}} \widehat{d\overline{z_k}^{N_k}}}{(|z_1^{N_1}|^2 + \dots + |z_n^{N_n}|^2)^n},
    \end{equation*}
    where $\widehat{d\overline{z_k}^{N_k}}$ means that $d\overline{z_k}^{N_k}$ is removed from
    $d\overline{z_1}^{N_1} \wedge \dots \wedge d\overline{z_n}^{N_n}$.
    Alternatively, using the expression \eqref{eq:artinian-ci-explicit}, we obtain the following expression,
\begin{equation} \label{eq:pairing-ljp2}
    \la g, \xi \ra_{Lu}(\beta) = \frac{1}{(2\pi i)^n} \int_{ \cap \{ |z_i^{N_i}| = \epsilon_i \}} \frac{\xi_0 a_n( h e_1 \wedge \dots \wedge e_n)}{z_1^{N_1}\dots z_n^{N_n}} \wedge \beta,
\end{equation}
    for any $(n,0)$-form $\beta$ which is holomorphic near $\{ 0 \}$.
    This type of explicit expression was used by Lejeune-Jalabert to obtain explicit expressions for the
    fundamental cycle of Artinian ideals, see \cite{LJ2}*{p. 239}.
\end{ex}


\begin{bibdiv}
\begin{biblist}

\bib{AndCH}{article}{
   author={Andersson, Mats},
   title={Uniqueness and factorization of Coleff-Herrera currents},
   journal={Ann. Fac. Sci. Toulouse Math.},
   volume={18},
   date={2009},
   number={4},
   pages={651--661},
   %issn={0240-2963},
}

\bib{AndNoeth}{article}{
   author={Andersson, Mats},
   title={Coleff-Herrera currents, duality, and Noetherian operators},
   %language={English, with English and French summaries},
   journal={Bull. Soc. Math. France},
   volume={139},
   date={2011},
   number={4},
   pages={535--554},
   %issn={0037-9484},
   %review={\MR{2869304 (2012k:32009)}},
}

\bib{AW1}{article}{
   author={Andersson, Mats},
   author={Wulcan, Elizabeth},
   title={Residue currents with prescribed annihilator ideals},
   %language={English, with English and French summaries},
   journal={Ann. Sci. \'Ecole Norm. Sup.},
   volume={40},
   date={2007},
   number={6},
   pages={985--1007},
   %issn={0012-9593},
   %review={\MR{2419855 (2010a:32001)}},
   %doi={10.1016/j.ansens.2007.11.001},
}

\bib{AW2}{article}{
   author={Andersson, Mats},
   author={Wulcan, Elizabeth},
   title={Decomposition of residue currents},
   journal={J. Reine Angew. Math.},
   volume={638},
   date={2010},
   pages={103--118},
   %issn={0075-4102},
   %review={\MR{2595337 (2011d:32013)}},
   %doi={10.1515/CRELLE.2010.004},
}

\bib{BS}{book}{
   author={B{\u{a}}nic{\u{a}}, Constantin},
   author={St{\u{a}}n{\u{a}}{\c{s}}il{\u{a}}, Octavian},
   title={Algebraic methods in the global theory of complex spaces},
   %note={Translated from the Romanian},
   publisher={Editura Academiei, Bucharest; John Wiley \& Sons, London-New
   York-Sydney},
   date={1976},
   pages={296},
   %review={\MR{0463470 (57 \#3420)}},
}

\bib{BjRings}{book}{
   author={Bj{\"o}rk, J.-E.},
   title={Rings of differential operators},
   series={North-Holland Mathematical Library},
   volume={21},
   publisher={North-Holland Publishing Co., Amsterdam-New York},
   date={1979}
   %pages={xvii+374},
   %isbn={0-444-85292-1},
   %review={\MR{549189 (82g:32013)}},
}

\bib{BjDmod}{article}{
   author={Bj{\"o}rk, Jan-Erik},
   title={$\scr D$-modules and residue currents on complex manifolds},
   status={Preprint, Stockholm},
   date={1996},
}

\bib{BjAbel}{article}{
   author={Bj{\"o}rk, Jan-Erik},
   title={Residues and $\scr D$-modules},
   conference={
      title={The legacy of Niels Henrik Abel},
   },
   book={
      publisher={Springer},
      place={Berlin},
   },
   date={2004},
   pages={605--651},
   %review={\MR{2077588 (2005f:32015)}},
}

\bib{BH}{book}{
   author={Bruns, Winfried},
   author={Herzog, J\"{u}rgen},
   title={Cohen-Macaulay rings},
   series={Cambridge Studies in Advanced Mathematics},
   volume={39},
   publisher={Cambridge University Press, Cambridge},
   date={1993},
   %pages={xii+403},
   %isbn={0-521-41068-1},
   %review={\MR{1251956}},
}

\bib{CH}{book}{
   author={Coleff, Nicolas R.},
   author={Herrera, Miguel E.},
   title={Les courants r\'esiduels associ\'es \`a une forme m\'eromorphe},
   %language={French},
   series={Lecture Notes in Mathematics},
   volume={633},
   publisher={Springer},
   place={Berlin},
   date={1978},
   %pages={x+211},
   %isbn={3-540-08651-X},
   %review={\MR{492769 (80j:32016)}},
}

\bib{DS1}{article}{
   author={Dickenstein, A.},
   author={Sessa, C.},
   title={Canonical representatives in moderate cohomology},
   journal={Invent. Math.},
   volume={80},
   date={1985},
   number={3},
   pages={417--434},
   %issn={0020-9910},
   %review={\MR{791667 (87a:32013)}},
   %doi={10.1007/BF01388723},
}

\bib{DS2}{article}{
   author={Dickenstein, Alicia},
   author={Sessa, Carmen},
   title={R\'esidus de formes m\'eromorphes et cohomologie mod\'er\'ee},
   %language={French, with English summary},
   conference={
      title={G\'eom\'etrie complexe},
      address={Paris},
      date={1992},
   },
   book={
      series={Actualit\'es Sci. Indust.},
      volume={1438},
      publisher={Hermann},
      place={Paris},
   },
   date={1996},
   pages={35--59},
   %review={\MR{1487899 (99a:32009)}},
}

\bib{Eis}{book}{
   author={Eisenbud, David},
   title={Commutative algebra},
   series={Graduate Texts in Mathematics},
   volume={150},
   note={With a view toward algebraic geometry},
   publisher={Springer-Verlag},
   place={New York},
   date={1995},
   %pages={xvi+785},
   %isbn={0-387-94268-8},
   %isbn={0-387-94269-6},
   %review={\MR{1322960 (97a:13001)}},
}

\bib{EHV}{article}{
   author={Eisenbud, David},
   author={Huneke, Craig},
   author={Vasconcelos, Wolmer},
   title={Direct methods for primary decomposition},
   journal={Invent. Math.},
   volume={110},
   date={1992},
   number={2},
   pages={207--235},
   %issn={0020-9910},
   %review={\MR{1185582 (93j:13032)}},
   %doi={10.1007/BF01231331},
}

\bib{FH}{article}{
   author={Fouli, Louiza},
   author={Huneke, Craig},
   title={What is a system of parameters?},
   journal={Proc. Amer. Math. Soc.},
   volume={139},
   date={2011},
   number={8},
   pages={2681--2696},
   %issn={0002-9939},
   %review={\MR{2801607 (2012e:13006)}},
   %doi={10.1090/S0002-9939-2011-10790-1},
}

\bib{GR}{book}{
   author={Grauert, Hans},
   author={Remmert, Reinhold},
   title={Coherent analytic sheaves},
   series={Grundlehren der Mathematischen Wissenschaften [Fundamental
   Principles of Mathematical Sciences]},
   volume={265},
   publisher={Springer-Verlag, Berlin},
   date={1984},
   %pages={xviii+249},
   %isbn={3-540-13178-7},
   %review={\MR{755331 (86a:32001)}},
   %doi={10.1007/978-3-642-69582-7},
}

\bib{GH}{book}{
   author={Griffiths, Phillip},
   author={Harris, Joseph},
   title={Principles of algebraic geometry},
   note={Pure and Applied Mathematics},
   publisher={Wiley-Interscience [John Wiley \& Sons]},
   place={New York},
   date={1978},
   %pages={xii+813},
   %isbn={0-471-32792-1},
   %review={\MR{507725 (80b:14001)}},
}

\bib{HLocal}{book}{
   author={Hartshorne, Robin},
   title={Local cohomology},
   series={A seminar given by A. Grothendieck, Harvard University, Fall},
   volume={1961},
   publisher={Springer-Verlag, Berlin-New York},
   date={1967},
   %pages={vi+106},
   %review={\MR{0224620 (37 \#219)}},
}

\bib{HGenDiv}{article}{
   author={Hartshorne, Robin},
   title={Generalized divisors on Gorenstein schemes},
   booktitle={Proceedings of Conference on Algebraic Geometry and Ring
   Theory in honor of Michael Artin, Part III (Antwerp, 1992)},
   journal={$K$-Theory},
   volume={8},
   date={1994},
   number={3},
   pages={287--339},
   %issn={0920-3036},
   %review={\MR{1291023}},
   %doi={10.1007/BF00960866},
}

\bib{HL}{book}{
   author={Huybrechts, Daniel},
   author={Lehn, Manfred},
   title={The geometry of moduli spaces of sheaves},
   series={Cambridge Mathematical Library},
   edition={2},
   publisher={Cambridge University Press, Cambridge},
   date={2010},
   %pages={xviii+325},
   %isbn={978-0-521-13420-0},
   %review={\MR{2665168}},
   %doi={10.1017/CBO9780511711985},
}

\bib{Iver}{book}{
   author={Iversen, Birger},
   title={Lecture notes on local rings},
   note={Edited and with a preface by Holger Andreas Nielsen},
   publisher={World Scientific Publishing Co. Pte. Ltd., Hackensack, NJ},
   date={2014},
   %pages={x+213},
   %isbn={978-981-4603-65-2},
   %review={\MR{3309563}},
   %doi={10.1142/9178},
}

\bib{24Loc}{book}{
   author={Iyengar, Srikanth B.},
   author={Leuschke, Graham J.},
   author={Leykin, Anton},
   author={Miller, Claudia},
   author={Miller, Ezra},
   author={Singh, Anurag K.},
   author={Walther, Uli},
   title={Twenty-four hours of local cohomology},
   series={Graduate Studies in Mathematics},
   volume={87},
   publisher={American Mathematical Society},
   place={Providence, RI},
   date={2007},
   %pages={xviii+282},
   %isbn={978-0-8218-4126-6},
   %review={\MR{2355715 (2009a:13025)}},
}

\bib{LJ1}{article}{
   author={Lejeune-Jalabert, Monique},
   title={Remarque sur la classe fondamentale d'un cycle},
   %language={French, with English summary},
   journal={C. R. Acad. Sci. Paris S\'er. I Math.},
   volume={292},
   date={1981},
   number={17},
   pages={801--804},
   %issn={0249-6321},
   %review={\MR{622423 (83d:32008)}},
}

\bib{LJ2}{article}{
   author={Lejeune-Jalabert, M.},
   title={Liaison et r\'esidu},
   %language={French},
   conference={
      title={Algebraic geometry},
      address={La R\'abida},
      date={1981},
   },
   book={
      series={Lecture Notes in Math.},
      volume={961},
      publisher={Springer, Berlin},
   },
   date={1982},
   pages={233--240},
   %review={\MR{708336 (85d:32021)}},
   %doi={10.1007/BFb0071285},
}

\bib{Lund1}{article}{
   author={Lundqvist, Johannes},
   title={A local Grothendieck duality theorem for Cohen-Macaulay ideals},
   journal={Math. Scand.},
   volume={111},
   date={2012},
   number={1},
   pages={42--52},
   %issn={0025-5521},
   %review={\MR{3001357}},
}

\bib{Lund2}{article}{
   author={Lundqvist, Johannes},
   title={A local duality principle for ideals of pure dimension},
   status={Preprint},
   date={2013},
   eprint={arXiv:1306.6252 [math.CV]},
   url={http://arxiv.org/abs/1306.6252},
}

\bib{LarComp}{article}{
   author={L\"ark\"ang, Richard},
   title={A comparison formula for residue currents},
   status={to appear},
   journal={Math. Scand.},
   %date={2012},
   eprint={arXiv:1207.1279 [math.CV]},
   url={http://arxiv.org/abs/1207.1279},
}

\bib{LarDua}{article}{
   author={L{\"a}rk{\"a}ng, Richard},
   title={On the duality theorem on an analytic variety},
   journal={Math. Ann.},
   volume={355},
   date={2013},
   number={1},
   pages={215--234},
   %issn={0025-5831},
   %review={\MR{3004581}},
   %doi={10.1007/s00208-012-0782-4},
}

\bib{LW1}{article}{
   author={L{\"a}rk{\"a}ng, Richard},
   author={Wulcan, Elizabeth},
   title={Computing residue currents of monomial ideals using comparison
   formulas},
   journal={Bull. Sci. Math.},
   volume={138},
   date={2014},
   number={3},
   pages={376--392},
   %issn={0007-4497},
   %review={\MR{3206474}},
   %doi={10.1016/j.bulsci.2013.06.003},
}

\bib{Markoe}{article}{
   author={Markoe, Andrew},
   title={A characterization of normal analytic spaces by the homological
   codimension of the structure sheaf},
   journal={Pacific J. Math.},
   volume={52},
   date={1974},
   pages={485--489},
   %issn={0030-8730},
   %review={\MR{0367262 (51 \#3504)}},
}

\bib{PMScand}{article}{
   author={Passare, Mikael},
   title={Residues, currents, and their relation to ideals of holomorphic
   functions},
   journal={Math. Scand.},
   volume={62},
   date={1988},
   number={1},
   pages={75--152},
   %issn={0025-5521},
   %review={\MR{961584 (90d:32019)}},
}

\bib{PS}{article}{
   author={Peskine, C.},
   author={Szpiro, L.},
   title={Liaison des vari\'et\'es alg\'ebriques. I},
   language={French},
   journal={Invent. Math.},
   volume={26},
   date={1974},
   pages={271--302},
   %issn={0020-9910},
   %review={\MR{0364271 (51 \#526)}},
}

\bib{PTY}{article}{
   author={Passare, Mikael},
   author={Tsikh, August},
   author={Yger, Alain},
   title={Residue currents of the Bochner-Martinelli type},
   journal={Publ. Mat.},
   volume={44},
   date={2000},
   number={1},
   pages={85--117},
   %issn={0214-1493},
   %review={\MR{1775747 (2001i:32006)}},
   %doi={10.5565/PUBLMAT_44100_02},
}

\bib{Roos}{article}{
   author={Roos, Jan-Erik},
   title={Bidualit\'e et structure des foncteurs d\'eriv\'es de $\varinjlim$
   dans la cat\'egorie des modules sur un anneau r\'egulier},
   %language={French},
   journal={C. R. Acad. Sci. Paris},
   volume={254},
   date={1962},
   pages={1720--1722},
   %review={\MR{0136640 (25 \#106b)}},
}

\bib{Sharp}{article}{
   author={Sharp, Rodney Y.},
   title={Dualizing complexes for commutative Noetherian rings},
   journal={Math. Proc. Cambridge Philos. Soc.},
   volume={78},
   date={1975},
   number={3},
   pages={369--386},
   %issn={0305-0041},
   %review={\MR{0382253}},
   %doi={10.1017/S0305004100051847},
}

\bib{ST}{book}{
   author={Siu, Yum-Tong},
   author={Trautmann, G{\"u}nther},
   title={Gap-sheaves and extension of coherent analytic subsheaves},
   series={Lecture Notes in Mathematics, Vol. 172},
   publisher={Springer-Verlag, Berlin-New York},
   date={1971},
   %pages={v+172},
   %review={\MR{0287033 (44 \#4240)}},
}

%\bib{Tong}{article}{
%   author={Tong, Yue Lin L.},
%   title={Integral representation formulae and Grothendieck residue symbol},
%   journal={Amer. J. Math.},
%   volume={95},
%   date={1973},
%   pages={904--917},
%   %issn={0002-9327},
%   %review={\MR{0367255 (51 \#3497)}},
%}

\bib{Vasc}{book}{
   author={Vasconcelos, Wolmer V.},
   title={Computational methods in commutative algebra and algebraic
   geometry},
   series={Algorithms and Computation in Mathematics},
   volume={2},
   %note={With chapters by David Eisenbud, Daniel R. Grayson, J\"urgen Herzog
   %and Michael Stillman},
   publisher={Springer-Verlag},
   place={Berlin},
   date={1998},
   %pages={xii+394},
   %isbn={3-540-60520-7},
   %review={\MR{1484973 (99c:13048)}},
   %doi={10.1007/978-3-642-58951-5},
}

\end{biblist}
\end{bibdiv}
\end{document}